\documentclass[11pt,letterpaper]{amsart}

\usepackage[utf8]{inputenc}
\usepackage{amsmath}
\usepackage{amssymb}
\usepackage{amsthm}
\usepackage{xparse}

\usepackage[pdftex]{hyperref}

\newtheorem{counter}{subcounter}[section]
\newtheorem{cor}[counter]{Corollary}
\newtheorem{lemma}[counter]{Lemma}
\newtheorem{prop}[counter]{Proposition}

\newtheorem{theorem}[counter]{Theorem}

\newtheorem{athm}{Theorem}

\theoremstyle{definition}
\newtheorem{defn}[counter]{Definition}
\newtheorem{example}[counter]{Example}
\newtheorem{remark}[counter]{Remark}

\newcommand{\eqnsp}[1]{\leavevmode\xleaders\hbox{\phantom{-}}\hfill\kern0pt\[#1\]}

\newcommand{\eps}{\varepsilon}

\newcommand{\hsp}{\mskip0.5\thinmuskip}
\newcommand{\msp}{\mskip0.25\thinmuskip}

\newcommand{\acdot}{\mskip0.75\thinmuskip\cdot\mskip0.75\thinmuskip}

\newcommand{\lsc}{\textsc{lsc}}
\newcommand{\usc}{\textsc{usc}}

\renewcommand{\Re}{\operatorname{Re}}

\newcommand{\A}{\mathcal{A}}
\newcommand{\B}{\mathcal{B}}

\newcommand{\D}{\mathcal{D}}

\renewcommand{\H}{\mathcal{H}}
\newcommand{\I}{\mathcal{I}}
\newcommand{\J}{\mathcal{J}}
\newcommand{\K}{\mathcal{K}}
\newcommand{\M}{\mathcal{M}}
\newcommand{\N}{\mathcal{N}}
\renewcommand{\P}{\mathcal{P}}
\newcommand{\Q}{\mathcal{Q}}

\renewcommand{\S}{\mathcal{S}}

\newcommand{\C}{\mathbb{C}}

\newcommand{\Fb}{\mathbb{F}}

\newcommand{\Nb}{\mathbb{N}}

\newcommand{\R}{\mathbb{R}}
\newcommand{\Tb}{\mathbb{T}}
\newcommand{\Zb}{\mathbb{Z}}

\DeclareMathOperator{\Aut}{Aut}
\DeclareMathOperator{\Prim}{Prim}
\DeclareMathOperator{\supp}{supp}
\DeclareMathOperator{\Span}{Span}

\NewDocumentCommand\abs{g}{
    \lvert\IfValueTF{#1}{#1}{\,\cdot\,}\rvert
}

\NewDocumentCommand\inner{gg}{
    \langle\IfValueTF{#2}{#1,#2}{\cdot,\IfValueTF{#1}{#1}{\cdot}}\rangle
}

\NewDocumentCommand\norm{g}{
    \lVert\IfValueTF{#1}{#1}{\,\cdot\,}\rVert
}

\calclayout

\hypersetup{
    colorlinks=true,
    linkcolor=blue,
    citecolor=red,
    urlcolor=cyan,
}

\begin{document}

\title{Minimality for noncommutative dynamical systems}

\author{Thomas Bray}
\address{University of Waterloo, 200 University Avenue West, Waterloo, Ontario, N2L 3G1, Canada}
\email{tgrbray@uwaterloo.ca}

\author{Erik S\'eguin}
\address{University of Waterloo, 200 University Avenue West, Waterloo, Ontario, N2L 3G1, Canada}
\email{erik.seguin@uwaterloo.ca}

\thanks{The first author was partially supported by an Ontario Graduate Scholarship.}

\begin{abstract}
We consider an alternative notion of minimality for C*-dynamical systems, which we refer to as strong minimality. This notion coincides with the usual definition of minimality in the commutative setting, but diverges sharply from it in the noncommutative setting in a way which we argue addresses several fundamental defects in the usual definition. We demonstrate that strong minimality is deeply connected with R\o rdam's notion of a C*-irreducible inclusion, giving a new perspective on the study of such inclusions. We provide various characterizations of this property, including for reduced crossed products, tensor products, and subhomogeneous C*-algebras. 
\end{abstract}

\maketitle

\setcounter{tocdepth}{1}
\tableofcontents

\section{Introduction}

Minimality is a fundamental notion in classical topological dynamics. A topological dynamical system ${(X,G,\alpha)}$ is \emph{minimal} if $X$ admits no proper nonempty $G$-invariant closed subsets. If $X$ is a compact Hausdorff space, then this property may also be characterized in terms of ${C(X)}$ through the duality between closed subsets of $X$ and closed ideals in ${C(X)}$. Taking the viewpoint inspired by Gelfand duality that C*-algebras are noncommutative topological spaces, dualizing minimality in this way suggests various possible corresponding noncommutative generalizations of the notion to C*-dynamical systems; it is one of these generalizations that we examine in this paper. 
\medskip

The widely accepted generalization of minimality to the noncommutative setting is as follows: a C*-dynamical system ${(\A,G,\alpha)}$ is said to be \emph{minimal} if $\A$ admits no nontrivial $G$-invariant closed two-sided ideals. This clearly extends the commutative notion, but it is rather unsatisfactory as a generalization to the noncommutative setting from the perspectives of both topological dynamics and operator algebras. 
\medskip

From a topological dynamical standpoint, the above definition of minimality for C*-dynamical systems fails as an extension of the commutative notion on the most basic level: it may not reflect the actual dynamics of the action. Indeed, it can fail rather catastrophically in this regard: if $\A$ is a simple C*-algebra, then every action is automatically minimal, even the trivial one. 
\medskip

This notion of minimality is also rather unsatisfactory from an operator algebraic perspective. Classically, the measurable dynamical counterpart to minimality is ergodicity; in particular, if $X$ is a finite set equipped with the discrete topology and the counting measure, minimality coincides with ergodicity for any action on $X$ (and both are equivalent to transitivity of the action). If one accepts the point of view that a finite-dimensional C*-algebra is the noncommutative analogue of a finite topological space, it is reasonable to expect that minimality should agree with ergodicity in this setting. However, this is not so: the C*-algebra of ${n\times n}$ matrices, ${M_n}$, is simple, therefore minimal irrespective of the action, but ergodic only if there are no projections fixed by the action. (In particular, the two do not coincide when ${n\geq2}$ and the action is trivial.)
\medskip

In this paper, we attempt to rectify these issues by modifying the usual definition of minimality in the noncommutative setting to preclude the presence of nontrivial $G$-invariant closed \emph{left} ideals rather than two-sided ideals. This seemingly slight modification turns out to be highly impactful: in addition to satisfactorily resolving the aforementioned issues, this alternative definition encodes operator algebraic structural properties which do not at first glance appear to have any connection with dynamics. 
\medskip

We say that a C*-dynamical system ${(\A,G,\alpha)}$ is \emph{strongly minimal} if $\A$ has no proper nonempty $G$-invariant closed left ideals. Strongly minimal C*-dynamical systems in the preceding sense have appeared previously in the literature: they were first introduced in \cite{ll75} (where they were given the name \emph{minimal}) and then further studied in \cite{avitzour84} and \cite{lp84}. However, the notion has since dropped out of consideration almost entirely, save for a few scattered mentions over the years. We attribute this disappearance to a number of reasons: direct competition with a similar definition, theorems with occasionally overly restrictive assumptions, and a shortage of interesting examples. We address all of these in the present work. 
\medskip

One of our main results is the following characterization of strong minimality. In the following statement, ${\partial_FG}$ denotes the Fustenberg boundary of $G$. 

\begin{athm}\label{strong-minimality-char}
Let ${(\A,G)}$ be a unital C*-dynamical system. The following are equivalent: 
\begin{enumerate}
\item Every $G$-invariant closed left ideal ${\I\subseteq\A}$ is trivial. 
\item Every $G$-invariant unital C*-subalgebra ${\B\subseteq\A}$ is minimal. 
\item Every $G$-equivariant unital completely positive linear map ${\Phi:\A\to C(\partial_FG)}$ is faithful. 
\item Every nonzero element ${x\in\A_+}$ admits a collection ${s_1,\dots,s_n\in G}$ (possibly with repetition) such that ${\smash{\sum_{j\,=\,1}^n s_j\msp x\geq1_\A}}$. 
\item Every nonzero open projection ${e\in\A^{**}}$ admits ${s_1,\dots,s_n\in G}$ such that ${\smash{\bigvee_{j\,=\,1}^n s_j\msp e=1_{\A^{**}}}}$. 
\item The Poisson map ${\P_\psi:\A\to\ell^\infty(G)}$ is faithful for every pure state ${\psi\in P(\A)}$. 
\end{enumerate}
\end{athm}

The subject of irreducible inclusions of C*-algebras has attracted considerable interest in recent years. In \cite{rordam23}, R\o rdam introduced the notion of C*-irreducibility for inclusions of C*-algebras. We show that the C*-irreducibility of an inclusion can be characterized by strong minimality of a corresponding C*-dynamical system. 

\begin{athm}\label{c*-irreducibility-char}
Let ${\B\subseteq\A}$ be a unital inclusion of C*-algebras. The following are equivalent: 
\begin{enumerate}
\item ${\B\subseteq\A}$ is C*-irreducible. 
\item ${(\A,G)}$ is strongly minimal for every subgroup ${G\leq U(\B)}$ with linear span norm dense in $\B$, where the action is by unitary conjugation. 
\item ${(\A,G)}$ is strongly minimal for some subgroup ${G\leq U(\B)}$ with linear span norm dense in $\B$, where the action is by unitary conjugation. 
\end{enumerate}
\end{athm}

In the second half of the paper, we examine strong minimality for reduced crossed products and tensor products. Our study of the former culminates in the following characterization. 

\begin{athm}\label{crossed-tensor-product-minimality-char}
Let ${(\A,G)}$ be a unital C*-dynamical system. The following are equivalent: 
\begin{enumerate}
\item ${(\A\rtimes_\lambda G,G)}$ is strongly minimal. 
\item ${(\A\otimes_\mathrm{min}C_\lambda^*(G),G)}$ is strongly minimal. 
\item $G$ is C*-simple and ${(\A,G)}$ is strongly minimal. 
\end{enumerate}
\end{athm}

We also show that strong minimality is well-behaved with respect to the natural product action, provided one of the C*-algebras underlying the component C*-dynamical systems is type I. 

\begin{athm}\label{tensor-product-minimality}
Let ${(\A,G)}$ and ${(\B,H)}$ be unital C*-dynamical systems. If $\A$ is a type I C*-algebra, then ${(\A,G)}$ and ${(\B,H)}$ are strongly minimal if and only if ${(\A\otimes_\mathrm{min}\B,G\times H)}$ is strongly minimal. 
\end{athm}

Finally, we completely characterize strong minimality in the subhomogeneous case. 

\begin{athm}\label{subhomogeneous-minimality}
Let ${(\A,G)}$ be a unital C*-dynamical system. If $\A$ is a subhomogeneous C*-algebra, then ${(\A,G)}$ is strongly minimal if and only if the following conditions hold: 
\begin{enumerate}
\item $\A$ is homogeneous; 
\item ${(\smash{\widehat{\A}},G)}$ is minimal, where ${\smash{\widehat{\A}}}$ is the spectrum of $\A$; and 
\item $\A$ has no nontrivial $G$-invariant projections. 
\end{enumerate}
\end{athm}

The remainder of this paper is organized as follows. In Section \ref{preliminary-section}, we fix our notation and recall the necessary preliminary material. In Section \ref{strong-minimality-section}, we contrast strong minimality with the accepted usual notion of minimality and prove various useful characterizations of the former. In Section \ref{c*-irreducibility-section}, we analyze the connection between C*-irreducible inclusions and strongly minimal C*-dynamical systems. In Section \ref{product-section}, we consider the behaviour of strongly minimal C*-dynamical systems with respect to taking tensor products and reduced crossed products. In Section \ref{example-section}, we examine various examples of strongly minimal C*-dynamical systems and obtain a characterization of this property for subhomogeneous flows. Finally, in Section \ref{non-unital-section}, we discuss our results in the non-unital setting. 
\medskip

\noindent\textbf{Acknowledgments.} The authors are grateful to Jashan Bal for mentioning the reference \cite{bb90}, to Matthew Kennedy for his support and encouragement, as well as for reading an earlier draft of the paper and providing helpful feedback, and to Jesse Peterson for asking a question that led to the consideration and inclusion of Remark \ref{non-type-I-product-minimality}. 

\section{Preliminaries}\label{preliminary-section}

\subsection{Notation}

In this subsection, we fix the notation that we will use throughout the paper. 
\medskip

If $X$ is a topological space, then we let ${\mathrm{Homeo}(X)}$ denote the group of homeomorphisms of $X$. Furthermore, if $X$ is compact and Hausdorff, then we let ${M_{1,+}(X)}$ denote the space of probability measures on $X$ and equip this space with the topology inherited from identifying it with the state space of ${C(X)}$ via the Riesz-Markov-Kakutani representation theorem (where the latter is endowed with the weak-* topology). 
\medskip

If $\H$ is a Hilbert space, then we let ${B(\H)}$ and ${K(\H)}$ denote the spaces of bounded and compact operators on $\H$, respectively. 
\medskip

If $\A$ is a C*-algebra, then we let ${S(\A)}$, ${P(\A)}$, and ${\Aut(\A)}$ denote the state space of $\A$, the set of pure states in ${S(\A)}$, and the automorphism group of $\A$, respectively. Moreover, we let ${\A_\mathrm{sa}}$ and ${\A_+}$ denote the set of self-adjoint elements in $\A$ and the set of positive elements in $\A$, respectively. We identify $\A$ with its image in the bidual ${\A^{**}}$ under the canonical embedding; abusively, we will suppress the notation specifying the embedding. We let ${\otimes_\mathrm{min}}$ denote the minimal tensor product of C*-algebras. Furthermore, if $\A$ is unital, then we let ${U(\A)}$ denote the unitary group of $\A$ and let ${\A^{**}_\lsc}$ and ${\A^{**}_\usc}$ denote the sets of lower and upper semi-continuous elements in ${\A^{**}}$, respectively (see Section \ref{nc-topology-subsection}). 
\medskip

If ${(\A,G,\alpha)}$ is a C*-dynamical system (see Section \ref{dynamics-subsection}) and ${X\subseteq\A}$, then we let ${X^G}$ denote the \emph{fixed point set} of $X$, i.e., the set defined by 
\eqnsp{X^G=\{x\in X:\alpha(s)(x)=x,\,\forall s\in G\}}

If $\A$ and $\B$ are C*-algebras and ${\varphi:\A\to\B}$ is a positive linear map, then we let ${\N_\varphi}$ denote the \emph{left kernel} of $\varphi$, i.e., the closed left ideal defined by 
\eqnsp{\N_\varphi=\{x\in\A:\varphi(x^*x)=0\}}

\subsection{Topological and noncommutative dynamics}\label{dynamics-subsection}

In this subsection, we give a brief overview of classical and noncommutative topological dynamics; the reader is referred to \cite{glasner76} for details on the former. 
\medskip

A \emph{topological dynamical system} ${(X,G,\alpha)}$ is a triple consisting of a topological space $X$ (always assumed nonempty herein), a discrete group $G$, and a group homomorphism ${\alpha:G\to\mathrm{Homeo}(X)}$. Henceforth, we omit the map $\alpha$ from the specification of such triples wherever so doing is unlikely to result in any ambiguity. Consistent with this omission, we will typically write simply ${sx}$ rather than ${\alpha(s)(x)}$, where ${s\in G}$ and ${x\in X}$. If $X$ is a compact Hausdorff space, then we refer to ${(X,G)}$ as a \emph{flow} and (in a slight abuse of terminology) refer to $X$ as a \emph{$G$-flow}. 
\medskip

If ${(X,G,\alpha)}$ and ${(Y,G,\beta)}$ are topological dynamical systems, then a map ${f:X\to Y}$ is said to be \emph{$G$-equivariant} if ${f\circ\alpha(s)=\beta(s)\circ f}$ for all ${s\in G}$. We say that $Y$ is a \emph{factor} of $X$ if there exists a continuous $G$-equivariant surjection of the latter onto the former. 
\medskip

If ${(X,G)}$ is a flow and ${Y\subseteq X}$ is a nonempty $G$-invariant closed subset, then $Y$ can be given the structure of a $G$-flow by endowing it with the restriction of the $G$-action, and in this case we say that $Y$ is a \emph{subflow} of $X$. 
\medskip

A topological dynamical system ${(X,G)}$ is \emph{minimal} if $X$ admits no nontrivial $G$-invariant closed subsets; equivalently, ${(X,G)}$ is minimal if ${\smash{\overline{Gx}}=X}$ for every ${x\in X}$. If ${(X,G)}$ is minimal, then in turn every factor of $X$ is also minimal; moreover, if ${f:Y\to X}$ is a continuous $G$-equivariant map, then $f$ is automatically surjective, and thus $X$ is a factor of $Y$. It is an easy consequence of Zorn's lemma that every flow admits a minimal subflow. 
\medskip

If $X$ is a $G$-flow, then there is a canonical induced action of $G$ on ${M_{1,+}(X)}$ endowing it with the structure of a $G$-flow given by 
\eqnsp{(s\mu)(f)=\mu(s^{-1}f)}
The flow $X$ is \emph{strongly proximal} if ${\smash{\overline{G\mu}}\cap\delta(X)\neq\emptyset}$ for every ${\mu\in M_{1,+}(X)}$, where ${\delta:X\hookrightarrow M_{1,+}(X)}$ denotes the embedding given by ${x\to\delta_x}$. 
\medskip

A flow ${(X,G)}$ is a \emph{boundary} if it is both minimal and strongly proximal; in this case, (in a slight abuse of terminology) we refer to $X$ as a \emph{$G$-boundary}. Every discrete group $G$ admits a boundary ${\partial_FG}$ which is universal in the sense that it factors onto every other $G$-boundary; this boundary is unique up to $G$-equivariant homeomorphism and is referred to as the \emph{Furstenberg boundary}. It was shown in \cite{kk17} that ${C(\partial_FG)}$ is isomorphic to the $G$-injective envelope of $\C$, where the latter is equipped with the trivial action (for the relevant details regarding injective envelopes, we refer the reader to \cite{hamana85}, \cite{bkko17}, \cite{kk17}, and \cite{kennedy20}). 
\medskip

A \emph{C*-dynamical system} ${(\A,G,\alpha)}$ is a triple consisting of a C*-algebra $\A$ (always assumed to be nonzero herein), a discrete group $G$, and a group homomorphism ${\alpha:G\to\Aut(\A)}$. As in the case of topological dynamical systems, we omit the map $\alpha$ from the specification of such triples where so doing is unlikely to result in any ambiguity and typically write simply ${sx}$ rather than ${\alpha(s)(x)}$, where ${s\in G}$ and ${x\in\A}$. If $\A$ is unital, then we also refer to ${(\A,G)}$ as a \emph{C*-flow}; moreover, if $\A$ is a W*-algebra, then we refer to ${(\A,G)}$ as a \emph{W*-dynamical system} or \emph{W*-flow}. 
\medskip

If ${(\A,G,\alpha)}$ and ${(\B,G,\beta)}$ are C*-dynamical systems, then a linear map ${\varphi:\A\to\B}$ is said to be \emph{$G$-equivariant} if it intertwines the actions $\alpha$ and $\beta$, i.e., if ${\varphi\circ\alpha(s)=\beta(s)\circ\varphi}$ for all ${s\in G}$. 
\medskip

If ${(\A,G)}$ is a (unital) C*-dynamical system and ${\B\subseteq\A}$ is a (unital) $G$-invariant C*-subalgebra, then we say that $\B$ is a \emph{factor} of $\A$. 

\subsection{Akemann's noncommutative topology}\label{nc-topology-subsection}

Throughout this subsection, we let $\A$ denote a unital C*-algebra. We summarize here the essential facts regarding the noncommutative topology introduced in \cite{akemann69}; the reader is referred to \cite{pedersen18} for a more comprehensive treatment of the elementary theory and to \cite{brown88} for a selection of more advanced topics. 
\medskip

An element ${x\in(\A^{**})_\mathrm{sa}}$ is said to be \emph{lower} \emph{semi-continuous} if ${\smash{x\rvert_{S(\A)}}}$ is a lower semi-continuous real-valued function; equivalently, $x$ is lower semi-continuous if there exists an increasing net ${(x_\alpha)}$ in ${\A_\mathrm{sa}}$ such that ${(x_\alpha)\nearrow x}$. Analogously, an element ${x\in(\A^{**})_\mathrm{sa}}$ is \emph{upper semi-continuous} if ${-x}$ is lower semi-continuous. A projection ${e\in\A^{**}}$ is said to be \emph{open} (respectively, \emph{closed}) if it is lower semi-continuous (respectively, upper semi-continuous) in this sense. The \emph{closure} of a projection is the least closed projection dominating it, and a projection is \emph{dense} if its closure is equal to ${1_{\A^{**}}}$. 
\medskip

Quite possibly the most important result in the noncommutative topology is the following triad of one-to-one correspondences: 
\begin{enumerate}
\item between norm closed left ideals of $\A$ and open projections in ${\A^{**}}$, given by 
\eqnsp{\I=\A^{**}e\cap\A\longleftrightarrow e=1_{\overline{\I}^\mathrm{w-*}}}
\item between weak-* closed faces of ${S(\A)}$ and closed projections in ${\A^{**}}$, given by 
\eqnsp{F=\{\psi\in S(\A):f(\psi)=1\}\longleftrightarrow f=\supp(F)}
\item between norm closed left ideals of $\A$ and weak-* closed faces of ${S(\A)}$, given by 
\eqnsp{\I=\bigcap_{\psi\,\in\,F}\N_\psi\longleftrightarrow F=\{\psi\in S(\A):\I\subseteq\N_\psi\}}
\end{enumerate}

\section{Strong minimality}\label{strong-minimality-section}

In this section, we recall the notion of strong minimality (originally introduced in \cite{ll75} under the name \emph{minimality}) and obtain various characterizations of this property. 

\begin{defn}
A C*-dynamical system ${(\A,G)}$ is \emph{strongly minimal} if $\A$ has no proper non-empty $G$-invariant closed left ideals. 
\end{defn}

\begin{remark}
If $\A$ is unital, then the assumption of closedness can be removed. 
\end{remark}

This should be compared with the usual notion of a minimal C*-dynamical system, which can be obtained by replacing left ideals with two-sided ideals in the above definition. The two notions evidently coincide when $\A$ is commutative; in this case, $\A$ is $G$-equivariantly isomorphic to ${C(X)}$ for some $G$-flow $X$ via Gelfand duality, and $\A$ is (strongly) minimal if and only if $X$ is minimal. However, the two notions can differ wildly when $\A$ is noncommutative, which is best illustrated by revisiting the examples from the introduction. 

\begin{example}
Let $\A$ be a simple unital C*-algebra and $G$ be a discrete group acting on $\A$ by the trivial action. It is clear in this case that ${(\A,G)}$ is minimal, since the simplicity of $\A$ implies that it admits no nontrivial closed two-sided ideals. However, nontrivial (unital) C*-algebras have an abundance of closed left ideals even when they are simple; in particular, ${(\A,G)}$ is strongly minimal if and only if ${\A\cong\C}$. 
\end{example}

Replacing minimality with strong minimality therefore resolves the topological dynamical issue presented by the failure of minimality to satisfactorily capture the dynamics of the action. 
\medskip

We recall here that a W*-flow ${(\M,G)}$ is said to be \emph{ergodic} if ${\M^G=\C}$ (i.e., if the action has no nontrivial fixed points). The next example shows that replacing minimality with strong minimality also resolves the operator algebraic issue presented by minimality not agreeing with ergodicity in the setting of ``noncommutative finite topological/measure spaces''. 

\begin{example}\label{motivating-example-ii}
Let $\A$ be a finite-dimensional C*-algebra and $G$ be a discrete group acting on $\A$. Let ${\I\subseteq\A}$ be a closed left ideal; then ${\I=\A e}$ for some (unique) projection ${e\in\A}$. It is easily seen that $\I$ is $G$-invariant if and only if $e$ is $G$-invariant, and therefore ${(\A,G)}$ is strongly minimal if and only if $\A$ has no nontrivial $G$-invariant projections, which is equivalent to ${(\A,G)}$ being ergodic. 
\end{example}

The correspondences outlined in Section \ref{nc-topology-subsection} preserve $G$-invariance going in every direction; we therefore easily obtain the following characterization of strong minimality by passing through these correspondences. 

\begin{prop}\label{strong-minimality-nc-topology-char}
Let ${(\A,G)}$ be a C*-flow. The following are equivalent: 
\begin{enumerate}
\item ${(\A,G)}$ is strongly minimal. 
\item There are no nontrivial $G$-invariant weak-* closed faces of ${S(\A)}$. 
\item There are no nontrivial $G$-invariant open (equivalently, closed) projections in ${\A^{**}}$. 
\end{enumerate}
\end{prop}

Simplicity for C*-algebras can be characterized either intrinsically (in terms of structure which is internal to the algebra, e.g., by the absence of nontrivial closed two-sided ideals) or extrinsically (in terms of the algebra's relations to other C*-algebras, e.g., by the faithfulness of every outgoing nonzero *-homomorphism). As (strong) minimality is a dynamical simplicity-type condition, it is natural to seek extrinsic characterizations analogous to those for simplicity. 
\medskip

Extrinsic characterizations in this sense have been obtained for the case where $G$ is amenable by Laison and Laison (\cite[Theorem 3]{ll75}), by Avitzour (\cite[Proposition 4.3]{avitzour84}), and by Longo and Peligrad (\cite[Lemma 2.5]{lp84}, \cite[Corollary 2.6]{lp84}). We extend these results beyond the amenable setting and obtain a general characterization by using the machinery of the injective envelope. 
\medskip

We will require the following lemma. 

\begin{lemma}\label{vanishing-equivariant-ucp}
Let ${(\A,G)}$ be a C*-flow. If ${\I\subseteq\A}$ is a proper closed $G$-invariant left ideal, then there exists a $G$-equivariant unital completely positive linear map ${\Phi:\A\to C(\partial_FG)}$ such that ${\I\subseteq\N_\Phi}$. 
\end{lemma}

\begin{proof}
As $\I$ is a proper closed left ideal, there exists a pure state ${\psi\in P(\A)}$ such that ${\I\subseteq\N_\psi}$ (this follows immediately from the correspondence between left ideals and weak-* closed faces outlined in Section \ref{nc-topology-subsection}). Let ${\S\subseteq\A}$ be the $G$-invariant operator system defined by 
\eqnsp{\S=\{x+y^*+\lambda\msp1_\A:x,y\in\I,\hsp\lambda\in\C\}}
The Cauchy-Schwarz inequality implies that ${\I\subseteq\ker(\psi)}$, and so ${\psi\rvert_\S}$ is $G$-invariant. Identifying $\C$ with its canonical image inside ${C(\partial_FG)}$, it then follows that $\psi$ extends to a $G$-equivariant unital completely positive linear map ${\Phi:\A\to C(\partial_FG)}$. 
\end{proof}

We are now prepared to prove our first extrinsic characterization of strong minimality. 

\begin{prop}\label{strong-minimality-extrinsic-char}
Let ${(\A,G)}$ be a C*-flow. The following are equivalent: 
\begin{enumerate}
\item ${(\A,G)}$ is strongly minimal. 
\item Every $G$-equivariant positive linear map ${\Phi:\A\to\B}$ is either faithful or identically zero, for every C*-flow ${(\B,G)}$. 
\item Every $G$-equivariant unital completely positive linear map ${\Phi:\A\to\B}$ is faithful, for every C*-flow ${(\B,G)}$. 
\item Every $G$-equivariant unital completely positive linear map ${\Phi:\A\to C(\partial_FG)}$ is faithful. 
\end{enumerate}
\end{prop}

\begin{proof}
Suppose that ${(1)}$ holds. Let ${(\B,G)}$ be a C*-flow and ${\Phi:\A\to\B}$ be a $G$-equivariant unital positive linear map; then ${\N_\Phi}$ is a closed $G$-invariant left ideal of $\A$, whence ${\N_\Phi=\{0\}}$ or ${\N_\Phi=\A}$. The former implies that $\Phi$ is faithful, while the latter implies that ${\Phi=0}$; the implication ${(1)\Rightarrow(2)}$ therefore holds. The implications ${(2)\Rightarrow(3)}$ and ${(3)\Rightarrow(4)}$ are clear. Suppose now that ${(4)}$ holds. If ${\I\subseteq\A}$ is a proper closed $G$-invariant left ideal, then it follows by Lemma \ref{vanishing-equivariant-ucp} that there exists a $G$-equivariant unital completely positive linear map ${\Phi:\A\to C(\partial_FG)}$ such that ${\I\subseteq\N_\Phi}$. As $\Phi$ is by assumption faithful, it follows that ${\I=\{0\}}$; thus ${(\A,G)}$ is strongly minimal, and therefore the implication ${(4)\Rightarrow(1)}$ holds. 
\end{proof}

When $G$ is amenable, the Furstenberg boundary is a singleton (see e.g., \cite[Theorem 3.1]{glasner76} or \cite[Corollary 3.17]{kk17}), and thus the equivalence ${(1)\Leftrightarrow(4)}$ in the above proposition recovers the previously obtained characterizations of \cite{ll75}, \cite{avitzour84}, and \cite{lp84}. 

\begin{remark}
It can be seen by looking at the proof of Lemma \ref{vanishing-equivariant-ucp} that ${C(\partial_FG)}$ may be replaced in Condition ${(4)}$ of the above proposition with $\B$ for any injective C*-flow ${(\B,G)}$. However, it can be shown that ${C(\partial_F G)}$ is essentially the heart of Condition ${(4)}$: if ${(\B,G)}$ is a C*-flow for which $\B$ can substitute for ${C(\partial_FG)}$ in Condition ${(4)}$, then ${C(\partial_FG)}$ embeds completely isometrically into $\B$ via a $G$-equivariant unital completely positive linear map. We give a short proof of this assertion. Let ${X=\partial_FG\times\partial_FG}$ be equipped with the natural diagonal $G$-action; then ${(X,G)}$ is not minimal unless ${\partial_FG}$ is a singleton, since the diagonal ${\{(x,x):x\in\partial_FG\}\subseteq X}$ is a subflow. If $G$ is amenable, then ${C(\partial_F G)\cong\C}$ and the claimed assertion is clear. Suppose that $G$ is not amenable; then ${\partial_FG}$ is not a singleton. This implies that ${(C(X),G)}$ is not strongly minimal, and therefore there exists a (non-faithful) $G$-equivariant unital completely positive linear map ${\Phi:C(X)\to\B}$. Composing $\Phi$ with the canonical embedding ${\pi:C(\partial_FG)\hookrightarrow C(\partial_FG)\otimes_\mathrm{min}C(\partial_FG)\cong C(X)}$ into the first tensor factor then yields a $G$-equivariant unital completely positive linear map ${\Psi:C(\partial_FG)\to\B}$, and it follows from the $G$-essentiality of the inclusion ${\C\subseteq C(\partial_FG)}$ that $\Psi$ is completely isometric. 
\end{remark}

It was shown in \cite[Proposition 4.2]{avitzour84} that strong minimality passes to factors. We can now give a short alternative proof of this fact using the above characterization of strong minimality. 

\begin{prop}\label{strong-minimality-hereditary}
Let ${(\A,G)}$ be a strongly minimal C*-flow. If ${\B\subseteq\A}$ is a factor, then ${(\B,G)}$ is strongly minimal. 
\end{prop}

\begin{proof}
Let ${\Phi:\B\to C(\partial_FG)}$ be any $G$-equivariant unital completely positive linear map; it is then an immediate consequence of the $G$-injectivity of ${C(\partial_FG)}$ that $\Phi$ extends to a $G$-equivariant unital completely positive linear map ${\Psi:\A\to C(\partial_FG)}$. The claim now follows from Proposition \ref{strong-minimality-extrinsic-char}. 
\end{proof}

The preceding proposition demonstrates that strong minimality is a hereditary property. In fact, it can be shown that hereditarity of minimality in the usual sense characterizes strong minimality. To prove this, we will require the following lemma. 

\begin{lemma}\label{completely-positive-map-faithful-domain}
Let $\A$ and $\B$ be C*-algebras. A contractive completely positive linear map ${\varphi:\A\to\B}$ is faithful if and only if it is faithful on its multiplicative domain. 
\end{lemma}

\begin{proof}
Let $\D$ denote the multiplicative domain of $\varphi$ and suppose that ${\varphi\rvert_\D}$ is faithful. Let ${x\in\N_\varphi}$; then ${\abs{x}\in\N_\varphi}$ as well, and thus it may be assumed without loss of generality that ${x\geq0}$. It follows by the Kadison-Schwarz inequality that 
\eqnsp{0\leq\varphi(x)^*\,\varphi(x)\leq\varphi(x^*x)=0}
and thus ${x\in\D}$, wherefore in turn ${\smash{x^{1/2}}\in\D}$ as well. The above equality also implies that ${\varphi(x)=0}$; it therefore follows that ${\smash{x^{1/2}}\in\N_\varphi}$ and so ${\smash{x^{1/2}}=0}$. This implies that ${x=0}$, and thus $\varphi$ is faithful. The converse is clear. 
\end{proof}

\begin{prop}\label{strong-minimality-hereditary-minimality-char}
Let ${(\A,G)}$ be a C*-flow. The following are equivalent: 
\begin{enumerate}
\item ${(\A,G)}$ is strongly minimal. 
\item Every factor ${\B\subseteq\A}$ is strongly minimal. 
\item Every factor ${\B\subseteq\A}$ is minimal. 
\end{enumerate}
\end{prop}

\begin{proof}
The implication ${(1)\Rightarrow(2)}$ is Proposition \ref{strong-minimality-hereditary}, while the implication ${(2)\Rightarrow(3)}$ is clear. The implication ${(3)\Rightarrow(1)}$ follows from Proposition \ref{strong-minimality-extrinsic-char} in conjunction with Lemma \ref{completely-positive-map-faithful-domain}. 
\end{proof}

Let ${(\A,G)}$ be a C*-flow. We recall here that the \emph{Poisson map} ${\P_\psi:\A\to\ell^\infty(G)}$ associated to a bounded linear functional ${\psi\in\A^*}$ is the $G$-equivariant bounded linear map defined by 
\eqnsp{\P_\psi(x)(s)=\psi(s^{-1}x)}
It is not difficult to see that ${\P_\psi}$ is unital (resp., positive) if and only if $\psi$ is unital (resp., positive). 
\medskip

When ${\A=C(X)}$ for some $G$-flow $X$, certain dynamical properties of ${(X,G)}$ are characterized by conditions on the Poisson maps: in particular, ${(X,G)}$ is minimal if and only if ${\P_\psi}$ is isometric for every pure state ${\psi\in P(\A)}$, while ${(X,G)}$ is a boundary if and only if ${\P_\psi}$ is isometric for every state ${\psi\in S(\A)}$. The former is equivalent to ${\P_\psi}$ being faithful for every pure state ${\psi\in P(\A)}$; this characterization of (strong) minimality carries over to the noncommutative setting. 

\begin{prop}\label{strong-minimality-poisson-map-char}
A C*-flow ${(\A,G)}$ is strongly minimal if and only if ${\P_\psi:\A\to\ell^\infty(G)}$ is faithful for every pure state ${\psi\in P(\A)}$. 
\end{prop}

\begin{proof}
Suppose that ${(\A,G)}$ is strongly minimal. Let ${\psi\in P(\A)}$; then ${\P_\psi}$ is a $G$-equivariant unital completely positive linear map, hence faithful by Proposition \ref{strong-minimality-extrinsic-char}. Conversely, suppose that ${\P_\psi}$ is faithful for every pure state ${\psi\in P(\A)}$. Let ${\I\subseteq\A}$ be a proper closed $G$-invariant left ideal; then there exists some ${\psi\in P(\A)}$ such that ${\I\subseteq\N_\psi}$, and it follows that 
\eqnsp{\P_\psi(x^*x)(s)=\psi(s^{-1}(x^*x))=\psi((s^{-1}x)^*\hsp(s^{-1}x))=0}
for all ${x\in\I}$ and ${s\in G}$. This implies that ${\P_\psi(x^*x)=0}$ for all ${x\in\I}$, wherefore ${\I=\{0\}}$. 
\end{proof}

\begin{remark}
The conclusion that ${\P_\psi}$ is faithful for every pure state ${\psi\in P(\A)}$ in the preceding proposition cannot be improved to ${\P_\psi}$ being isometric as in the commutative setting. Indeed, if $G$ is a discrete group and ${\pi:G\to U(M_n)}$ is an irreducible representation, then Example \ref{motivating-example-ii} implies that ${(M_n,G)}$ is strongly minimal, where the action is by conjugation, but if $G$ is finite and ${n\geq2}$, then ${\P_\psi}$ is not isometric for any pure state ${\psi\in P(M_n)}$: if it were, this would yield a finite set of vector states which norms ${M_n}$, which is impossible. 
\end{remark}

The characterizations of strong minimality we have obtained thus far have largely been ``global'' in nature, in the sense that they are framed in terms of the overall structure of the C*-flow itself or in terms of its relation to other C*-flows. It is convenient to also have characterizations which are ``local'', i.e., which are framed in terms of a single element. One characterization of this type is the fixed point property of the equivalence ${(2)\Leftrightarrow(3)}$ in \cite[Proposition 4.1]{avitzour84}; we record below the statement and proof of this result for the reader's convenience. 

\begin{prop}\label{strong-minimality-fixed-point-char}
A C*-flow ${(\A,G)}$ is strongly minimal if and only if ${(\A^{**}_\lsc)^G=\R}$. 
\end{prop}

\begin{proof}
Suppose that ${(\A,G)}$ is strongly minimal. Let ${f\in(\A^{**}_\lsc)^G}$, let ${\lambda=\inf\{f(\psi):\psi\in S(\A)\}}$, and let ${F\subseteq S(\A)}$ be the nonempty weak-* closed face defined by 
\eqnsp{F=\{\psi\in S(\A):f(\psi)\leq\lambda\}}
As $F$ is $G$-invariant, Proposition \ref{strong-minimality-nc-topology-char} implies that ${F=S(\A)}$; then ${f(\psi)=\lambda}$ for all ${\psi\in S(\A)}$, and thus it follows that ${f=\lambda\msp1_{\A^{**}}}$. Conversely, if ${(\A^{**}_\lsc)^G=\R}$, then Proposition \ref{strong-minimality-nc-topology-char} evidently implies that ${(\A,G)}$ is strongly minimal, since open projections in the bidual are lower semi-continuous. 
\end{proof}

\begin{cor}\label{strong-minimality-fixed-point-cor}
If ${(\A,G)}$ is a strongly minimal C*-flow, then ${\A^G=\C}$. In particular, if ${(\M,G)}$ is a strongly minimal W*-flow, then ${(\M,G)}$ is ergodic. 
\end{cor}

\begin{remark}\label{ergodic-non-minimal-remark}
The converse of Corollary \ref{strong-minimality-fixed-point-cor} is false even when $\A$ is commutative. This can be seen through the following example: let $G$ be an infinite discrete group, let ${\A=\ell^\infty(G)}$, and let $G$ act on $\A$ by left translation. It is clear that ${\A^G=\C}$, but ${c_0(G)}$ is a nontrivial closed $G$-invariant two-sided ideal, which presents an obvious obstruction to (strong) minimality of ${(\ell^\infty(G),G)}$. 
\end{remark}

While Proposition \ref{strong-minimality-fixed-point-char} gives a useful local characterization of strong minimality, we would like ideally to also have a characterization on the level of $\A$ itself rather than its bidual. Remark \ref{ergodic-non-minimal-remark} demonstrates that hoping for a simple fixed point characterization on this level is somewhat overly optimistic. Instead, we draw inspiration from R\o rdam's notion of relative fullness (see \cite{rordam23}) and introduce a natural dynamical extension of the same. 

\begin{defn}
Let ${(\A,G)}$ be a C*-flow. An element ${x\in\A_+}$ is \emph{G-full} if there exists a collection of elements ${s_1,\dots,s_n\in G}$ (possibly with repetition) such that ${\sum_{j\,=\,1}^n s_j\msp x\geq1_\A}$. 
\end{defn}

\begin{remark}
It is easily seen that ${x\in\A_+}$ is $G$-full if and only if there exist ${s_1,\dots,s_n\in G}$ such that ${\sum_{j\,=\,1}^n s_j\msp x}$ is invertible (cf. \cite[Lemma 3.5]{rordam23}). 
\end{remark}

The next lemma gives the essential link between our local condition ($G$-fullness) and our global condition (strong minimality). 

\begin{prop}\label{strong-minimality-local-char-lemma}
Let ${(\A,G)}$ be a C*-flow. An element ${x\in\A_+}$ is $G$-full if and only if the closed $G$-invariant left ideal generated by $x$ is equal to $\A$. 
\end{prop}

\begin{proof}
Let ${\I\subseteq\A}$ denote the closed $G$-invariant left ideal generated by $x$. If $x$ is $G$-full, then there exist ${s_1,\dots,s_n\in G}$ such that ${\sum_{j\,=\,1}^n s_j\msp x}$ is invertible; therefore $\I$ contains an invertible element, whence ${\I=\A}$. Conversely, suppose that ${\I=\A}$. Let ${\J\subseteq\A}$ be the left ideal defined by 
\eqnsp{\J=\Span(\{w\hsp(sx^{1/2}):w\in\A,\hsp s\in G\})}
As ${\I\subseteq\smash{\overline{\J}}}$, there exists an invertible element ${y\in\J}$. Now let ${s_1,\dots,s_n\in G}$ and ${w_1,\dots,w_n\in\A}$ be elements such that 
\eqnsp{y=\sum_{j\,=\,1}^n w_j\hsp(s_j\msp x^{1/2})}
For ${j,k=1,\dots,n}$, let ${z_{j,k}=w_k^*\hsp w_j}$; then 
\eqnsp{(z_{j,k}\hsp(s_j\msp x^{1/2})-s_k\msp x^{1/2})^*\hsp(z_{j,k}\hsp(s_j\msp x^{1/2})-s_k\msp x^{1/2})\geq0}
As ${(z_{j,k})^*=z_{k,j}}$, it follows by expanding and rearranging the above inequality that 
\begin{align*}
(s_j\msp x^{1/2})\hsp z_{k,j}\hsp(s_k\msp x^{1/2})+(s_k\msp x^{1/2})\hsp z_{j,k}\hsp(s_j\msp x^{1/2})&\leq(s_j\msp x^{1/2})\hsp\abs{z_{j,k}}^2\hsp(s_j\msp x^{1/2})+s_k\msp x
\\&\leq\norm{z_{j,k}}^2\hsp s_j\msp x+s_k\msp x
\end{align*}
for ${j,k=1,\dots,n}$, whence 
\eqnsp{y^*y=\Re(y^*y)=\sum_{j,k\,=\,1}^n\Re((s_k\msp x^{1/2})\hsp z_{j,k}\hsp(s_j\msp x^{1/2}))\leq\frac{1}{2}\sum_{j,k\,=\,1}^n(\norm{z_{j,k}}^2\hsp s_j\msp x+s_k\msp x)}
As ${y^*y}$ is positive and invertible, there exists ${\eps>0}$ such that ${y^*y\geq\eps\msp1_\A}$; the above inequality then implies that ${\smash{\sum_{j,k\,=\,1}^n(\norm{z_{j,k}}^2\hsp s_j\msp x+s_k\msp x)}}$ is invertible as well, whence $x$ is $G$-full. 
\end{proof}

This yields our desired local characterization of strong minimality. 

\begin{prop}\label{strong-minimality-local-char}
Let ${(\A,G)}$ be a C*-flow. The following are equivalent: 
\begin{enumerate}
\item ${(\A,G)}$ is strongly minimal. 
\item Every nonzero positive element ${x\in\A_+}$ is $G$-full.  
\item Every nonzero open projection ${e\in\A^{**}}$ admits ${s_1,\dots,s_n\in G}$ such that ${\smash{\bigvee_{j\,=\,1}^n s_j\msp e=1_{\A^{**}}}}$. 
\end{enumerate}
\end{prop}

\begin{proof}
The implication ${(1)\Rightarrow(2)}$ follows from Lemma \ref{strong-minimality-local-char-lemma}, while the implication ${(3)\Rightarrow(1)}$ follows from Proposition \ref{strong-minimality-nc-topology-char}. Suppose that ${(2)}$ holds and let ${e\in\A^{**}}$ be a nonzero open projection; then there exists a nonzero element ${x\in\A_+}$ such that ${x\leq e}$. As $x$ is $G$-full, there exist ${s_1,\dots,s_n\in G}$ such that ${\smash{\sum_{j\,=\,1}^n s_j\msp x\geq1_\A}}$. This implies ${\smash{\sum_{j\,=\,1}^n s_j\msp e\geq1_{\A^{**}}}}$, whence 
\eqnsp{0\leq1_{\A^{**}}-\bigvee_{j\,=\,1}^n s_j\msp e\leq\Big(1_{\A^{**}}-\bigvee_{j\,=\,1}^n s_j\msp e\Big)\Big(\sum_{j\,=\,1}^n s_j\msp e\Big)\Big(1_{\A^{**}}-\bigvee_{j\,=\,1}^n s_j\msp e\Big)=0}
It therefore follows that ${\bigvee_{j\,=\,1}^n s_j\msp e=1_{\A^{**}}}$, which proves the implication ${(2)\Rightarrow(3)}$. 
\end{proof}

\begin{remark}\label{open-proj-covers-one}
The implication ${(1)\Rightarrow(3)}$ in the above proposition can also be proved directly by use of the noncommutative topology on ${\A^{**}}$: if ${e\in\A^{**}}$ is a nonzero open projection, then ${\bigvee_{s\,\in\,G}se}$ is a nonzero $G$-invariant open projection, hence equal to ${1_{\A^{**}}}$ by Proposition \ref{strong-minimality-nc-topology-char}, and it follows by \cite[Proposition II.10]{akemann69} that ${\bigvee_{j\,=\,1}^n s_j\msp e=1_{\A^{**}}}$ for some ${s_1,\dots,s_n\in G}$. 
\end{remark}

\begin{remark}
If $G$ acts on a Boolean algebra $B$ by automorphisms, then the action is minimal if for every nonzero element ${x\in B}$ there exist ${s_1,\dots,s_n\in G}$ such that ${\smash{\bigvee_{j\,=\,1}^n s_j\msp x=1_B}}$ (see \cite{bb90}). The open projections in the bidual of a C*-algebra do not generally form a Boolean algebra, since the meet of two open projections may fail to be open (\cite[Example II.6]{akemann69}), but Condition ${(3)}$ in the preceding proposition is nevertheless clearly a minimality condition of this type on the set of open projections in the bidual. 
\end{remark}

We now obtain Theorem \ref{strong-minimality-char} by splicing together the characterizations of strong minimality that were proved in this section. 

\begin{proof}[Proof of Theorem \ref{strong-minimality-char}]
The claim follows from the combination of Proposition \ref{strong-minimality-extrinsic-char}, Proposition \ref{strong-minimality-hereditary-minimality-char}, Proposition \ref{strong-minimality-poisson-map-char}, and Proposition \ref{strong-minimality-local-char}. 
\end{proof}

We conclude this section by showing that non-full elements can be separated from the identity by boundary maps. 

\begin{prop}
Let ${(\A,G)}$ be a C*-flow. An element ${x\in\A_+}$ is $G$-full if and only if ${\Phi(x)\neq0}$ for every $G$-equivariant unital positive linear map ${\Phi:\A\to C(\partial_FG)}$. Moreover, in this case there exists ${\eps>0}$ such that ${\norm{\Phi(x)}_\infty\geq\eps}$ for every such map ${\Phi:\A\to C(\partial_FG)}$. 
\end{prop}

\begin{proof}
Suppose first that $x$ is $G$-full; then there exist some collection ${s_1,\dots,s_n\in G}$ (possibly with repetition) such that ${\smash{\sum_{j\,=\,1}^n s_j\msp x\geq1_\A}}$. If ${\Phi:\A\to C(\partial_FG)}$ is a $G$-equivariant unital positive linear map, then 
\eqnsp{\sum_{j\,=\,1}^n s_j\hsp\Phi(x)=\Phi\Big(\sum_{j\,=\,1}^n s_j\msp x\Big)\geq\chi_{\partial_FG}}
and so ${\norm{\Phi(x)}_\infty\geq\frac{1}{n}}$. Conversely, suppose that $x$ is not $G$-full; then it follows by Lemma \ref{strong-minimality-local-char-lemma} that the closed $G$-invariant left ideal generated by $x$ is not equal to $\A$, whence Lemma \ref{vanishing-equivariant-ucp} implies that there exists a $G$-equivariant unital positive linear map ${\Phi:\A\to C(\partial_FG)}$ such that ${\Phi(x)=0}$. 
\end{proof}

\begin{remark}
When $G$ is the trivial group, $G$-fullness coincides with invertibility, and the above proposition reduces to the usual separation of non-invertible elements from the identity by states. 
\end{remark}

\section{C*-irreducibility}\label{c*-irreducibility-section}

In this section, we recall R\o rdam's notion of C*-irreducibility from \cite{rordam23} and show it admits a natural characterization in terms of strong minimality. 
\medskip

A unital inclusion of C*-algebras ${\B\subseteq\A}$ is \emph{C*-irreducible} if every C*-algebra $\D$ intermediate to the inclusion (in the sense that ${\B\subseteq\D\subseteq\A}$) is simple. The notion is inspired by the W*-algebraic notion of an irreducible inclusion: an inclusion ${\N\subseteq\M}$ of von Neumann algebras is \emph{irreducible} if every von Neumann algebra $\Q$ intermediate to the inclusion is a factor. It is not difficult to show that this is equivalent to the relative commutant being trivial. If the group ${U(\N)}$ acts on $\M$ by conjugation, then 
\eqnsp{\M^{U(\N)}=U(\N)'\cap\M=\N'\cap\M}
which clearly yields the following dynamical characterization: the inclusion ${\N\subseteq\M}$ is irreducible if and only if the W*-flow ${(\M,U(\N))}$ is ergodic. In light of this, it is not unreasonable to expect that C*-irreducibility admits an analogous dynamical characterization. We now demonstrate that this is indeed the case. 

\begin{proof}[Proof of Theorem \ref{c*-irreducibility-char}]
The implication ${(2)\Rightarrow(3)}$ is clear. The implications ${(1)\Rightarrow(2)}$ and ${(3)\Rightarrow(1)}$ follow by Proposition \ref{strong-minimality-local-char} along with \cite[Lemma 3.5]{rordam23} and \cite[Proposition 3.7]{rordam23}. 
\end{proof}

Setting ${\B=\A}$ and ${G=U(\A)}$ recovers the result of \cite[Proposition 2.15]{lp84}. 

\begin{cor}
A unital C*-algebra $\A$ is simple if and only if ${(\A,U(\A))}$ is strongly minimal. 
\end{cor}

We also obtain the following result concerning inclusions of group C*-algebras. 

\begin{cor}\label{group-c*-irreducibility-char}
Let ${H\leq G}$ be an inclusion of discrete groups. If ${\pi:G\to U(\H)}$ is a representation and ${\sigma:H\to U(\H)}$ is the restriction of $\pi$ to $H$, then ${C_\sigma^*(H)\subseteq C_\pi^*(G)}$ is C*-irreducible if and only if ${(C_\pi^*(G),H)}$ is strongly minimal, where the action is by conjugation via the unitaries in ${\pi(H)}$. 
\end{cor}

\begin{remark}
The case where ${\pi=\lambda}$ (the left regular representation of $G$) is of particular interest. As ${\lambda\rvert_H}$ is weakly equivalent to the left regular representation of $H$, by combining Theorem \ref{strong-minimality-char} and Corollary \ref{group-c*-irreducibility-char} we obtain the following equivalence: the inclusion ${C_\lambda^*(H)\subseteq C_\lambda^*(G)}$ is C*-irreducible if and only if every $H$-equivariant unital (completely) positive linear map ${\Phi:C_\lambda^*(G)\to C(\partial_FH)}$ is faithful. As far as we are aware, this is the first characterization of C*-irreducibility for inclusions of reduced group C*-algebras in terms of a condition involving the Furstenberg boundary, vaguely reminiscent of the results of \cite{kk17} characterizing C*-simplicity. 
\end{remark}

We record explicitly the following important special case of Corollary \ref{group-c*-irreducibility-char}. 

\begin{cor}\label{group-c*-simple-char}
Let $G$ be a discrete group. If ${\pi:G\to U(\H)}$ is a representation, then ${C_\pi^*(G)}$ is simple if and only if ${(C_\pi^*(G),G)}$ is strongly minimal. 
\end{cor}

\begin{remark}
For ${\pi\neq\lambda}$ (the left regular representation), this is again seemingly the first instance of a condition characterizing simplicity of a group C*-algebra arising from a representation other than the left regular representation in terms of the Furstenberg boundary. 
\end{remark}

We conclude this section by examining the connection between C*-irreducibility and two other irreducibility-type properties for unital inclusions of simple C*-algebras. A unital inclusion ${\B\subseteq\A}$ of C*-algebras is said to be \emph{relatively simple} (following \cite{ursu22}) if every unital completely positive linear map ${\varphi:\A\to B(\H)}$ such that ${\varphi\rvert_\B}$ is a *-homomorphism is faithful, and \emph{hereditarily essential} (following \cite{pz15}) if the inclusion ${\B\subseteq\D}$ is essential for every unital C*-algebra $\D$ intermediate to the inclusion ${\B\subseteq\A}$. 

\begin{prop}
Let ${\B\subseteq\A}$ be a unital inclusion of C*-algebras. The following are equivalent: 
\begin{enumerate}
\item ${\B\subseteq\A}$ is C*-irreducible. 
\item ${\B\subseteq\A}$ is relatively simple. 
\item ${\B\subseteq\A}$ is hereditarily essential and $\B$ is simple. 
\end{enumerate}
\end{prop}

\begin{proof}
The equivalence ${(1)\Leftrightarrow(3)}$ is clear, while the implication ${(2)\Rightarrow(1)}$ is contained in \cite[Proposition 3.6]{ursu22}. Suppose now that ${(1)}$ holds and let ${\varphi:\A\to B(\H)}$ be a unital completely positive linear map such that ${\varphi\rvert_\B}$ is a *-homomorphism. Let ${\D\subseteq\A}$ denote the multiplicative domain of $\varphi$. As ${\varphi\rvert_\B}$ is a *-homomorphism, evidently ${\B\subseteq\D}$; it thus follows that $\D$ is a C*-algebra intermediate to the inclusion ${\B\subseteq\A}$ and consequently simple. As ${\varphi\rvert_\D}$ is a unital *-homomorphism, this implies that ${\varphi\rvert_\D}$ is faithful; it then follows by Lemma \ref{completely-positive-map-faithful-domain} that $\varphi$ is faithful, which proves the remaining implication ${(1)\Rightarrow(2)}$. 
\end{proof}

\section{Crossed products and tensor products}\label{product-section}

In this section, we study the strong minimality of tensor products and reduced crossed products. We characterize when reduced crossed products are strongly minimal and demonstrate that strong minimality is stable with respect to taking (minimal) tensor products when one of the component C*-flows has an underlying type I C*-algebra. Additionally, we characterise strong minimality of diagonal C*-flows when one of the component C*-flows arises from a reduced group C*-algebra. 
\medskip

The next proposition is a reframing of \cite[Theorem 5.12]{rordam23} in the setup of strong minimality. We offer a different proof which shows the utility of the connection between strong minimality and the Furstenberg boundary. 
\medskip

We will require the following preliminary lemma; the proof is a straightforward modification of that of \cite[Proposition 3.3]{kennedy20} (and indeed this result can clearly be recovered by taking ${\A=\C}$ in the statement of the lemma). 

\begin{lemma}\label{c*-simple-crossed-product-boundary-map-factors}
Let ${(\A,G)}$ be a C*-flow. If $G$ is C*-simple, then every $G$-equivariant unital positive linear map ${\Phi:\A\rtimes_\lambda G\to C(\partial_FG)}$ factors through the canonical expectation ${E:\A\rtimes_\lambda G\to\A}$. 
\end{lemma}

\begin{proof}
Let ${\B=(C(\partial_FG)\otimes_\mathrm{min}\A)\rtimes_\lambda G}$. Identifying ${\A\rtimes_\lambda G}$ with its image in $\B$ under the natural inclusion, the $G$-injectivity of ${C(\partial_FG)}$ implies that $\Phi$ extends to a $G$-equivariant unital completely positive linear map ${\Psi:\B\to C(\partial_FG)}$. Identifying ${C(\partial_FG)}$ with its image in $\B$ under the canonical inclusion, the $G$-rigidity of the inclusion ${\C\subseteq C(\partial_FG)}$ implies that ${\Psi\rvert_{C(\partial_FG)}}$ is equal to the identity map, and therefore in particular that ${C(\partial_FG)}$ is contained in the multiplicative domain of $\Psi$. Now fix elements ${x\in\A}$ and ${s\in G\setminus\{e\}}$. Let ${\eta\in\partial_FG}$; then it follows by \cite[Theorem 3.1]{bkko17} and Urysohn's lemma that there exists ${f\in C(\partial_FG)}$ such that ${f(s^{-1}\eta)=0}$ and ${f(\eta)=1}$, and so 
\begin{align*}
\Phi(x\lambda_s)(\eta)&=\Psi((1_{C(\partial_FG)}\otimes x)\,\lambda_s)(\eta)
\\&=\Psi((1_{C(\partial_FG)}\otimes x)\,\lambda_s)(\eta)\,f(\eta)
\\&=\Psi((1_{C(\partial_FG)}\otimes x)\,\lambda_s)(\eta)\,\Psi(f\otimes1_\A)(\eta)
\\&=\Psi((1_{C(\partial_FG)}\otimes x)\,\lambda_s\hsp(f\otimes1_\A))(\eta)
\\&=\Psi((1_{C(\partial_FG)}\otimes x)(sf\otimes1_\A)\,\lambda_s)(\eta)
\\&=\Psi(sf\otimes1_\A)(\eta)\,\Psi((1_{C(\partial_FG)}\otimes x)\,\lambda_s)(\eta)
\end{align*}
As ${\Psi(sf\otimes1_\A)=f(s^{-1}\eta)=0}$, this implies that ${\Phi(x\lambda_s)(\eta)=0}$; as this holds for every ${\eta\in\partial_FG}$, it follows that ${\Phi(x\lambda_s)=0}$. This implies that $\Phi$ agrees with ${\Phi\circ E}$ on ${\Span(\{x\lambda_s:x\in\A,\hsp s\in G\})}$; by taking norm limits it then follows that ${\Phi=\Phi\circ E}$, which proves the claim. 
\end{proof}

We are now prepared to prove the aforementioned proposition. 

\begin{prop}\label{strongly-minimal-crossed-product-char}
If ${(\A,G)}$ is a C*-flow, then ${(\A\rtimes_\lambda G,G)}$ is strongly minimal if and only if $G$ is C*-simple and ${(\A,G)}$ is strongly minimal. 
\end{prop}

\begin{proof}
Suppose first that ${(\A\rtimes_\lambda G,G)}$ is strongly minimal; then Proposition \ref{strong-minimality-hereditary} implies that ${(\A,G)}$ and ${(C_\lambda^*(G),G)}$ are strongly minimal. By Proposition \ref{group-c*-simple-char}, the latter implies that $G$ is C*-simple, which proves the forward direction. Conversely, suppose that $G$ is C*-simple and ${(\A,G)}$ is strongly minimal. Let ${\Phi:\A\rtimes_\lambda G\to C(\partial_FG)}$ be a $G$-equivariant unital positive linear map; then it follows by Lemma \ref{c*-simple-crossed-product-boundary-map-factors} that ${\Phi=\Phi\rvert_\A\circ E}$, where ${E:\A\rtimes_\lambda G\to\A}$ is the canonical expectation. As ${(\A,G)}$ is strongly minimal, Theorem \ref{strong-minimality-char} implies that ${\Phi\rvert_\A}$ is faithful; $\Phi$ is thus the composition of faithful linear maps, therefore itself faithful, and so ${(\A\rtimes_\lambda G,G)}$ is strongly minimal by Theorem \ref{strong-minimality-char}. 
\end{proof}

\begin{remark}
The backward direction in the above proposition can alternatively be deduced from \cite[Theorem 5.12]{rordam23} as follows: Proposition \ref{strong-minimality-local-char} implies that Condition ${(2)}$ in \cite[Theorem 5.12]{rordam23} holds, and thus ${C_\lambda^*(G)\subseteq\A\rtimes_\lambda G}$ is C*-irreducible. The claim then follows from Theorem \ref{c*-irreducibility-char}. 
\end{remark}

If ${(\A,G,\alpha)}$ is a C*-flow and $F$ is a discrete group which quotients onto $G$, then we can define a new C*-flow ${(\A,F,\beta)}$ by composing the quotient map with $\alpha$. As $\beta$ factors through $\alpha$, it is easily seen that ${(\A,F)}$ is strongly minimal if and only if ${(\A,G)}$ is strongly minimal. By taking ${F=\Fb_G}$ (the free group with basis $G$), we may make use of Proposition \ref{strongly-minimal-crossed-product-char} even when $G$ is not C*-simple. This is exemplified by the following proposition. 

\begin{prop}\label{strong-minimality-c*-irreducibility-char}
A C*-flow ${(\A,G)}$ is strongly minimal if and only if the inclusion 
\eqnsp{C_\lambda^*(\Fb_G)\subseteq\A\rtimes_\lambda\Fb_G}
is C*-irreducible. 
\end{prop}

\begin{proof}
As ${\Fb_G}$ is C*-simple, Proposition \ref{strongly-minimal-crossed-product-char} implies that ${(\A,\Fb_G)}$ is strongly minimal if and only if ${(\A\rtimes_\lambda\Fb_G,\Fb_G)}$ is strongly minimal; the claimed equivalence then follows from Theorem \ref{c*-irreducibility-char}. 
\end{proof}

\begin{remark}
In light of Theorem \ref{c*-irreducibility-char} and Proposition \ref{strong-minimality-c*-irreducibility-char}, we see that in some sense the question of whether a particular unital inclusion of C*-algebras is C*-irreducible is dual to the question of whether or not a particular C*-flow is strongly minimal. 
\end{remark}

If ${(\A,G)}$ and ${(\B,H)}$ are C*-flows, then there is a natural product action of ${G\times H}$ on ${\A\otimes_\mathrm{min}\B}$ defined by 
\eqnsp{(s,t)\cdot(x\otimes y)=(s\cdot x)\otimes(t\cdot y)}
In the commutative case, where ${\A=C(X)}$ for some $G$-flow $X$ and ${\B=C(Y)}$ for some $H$-flow $Y$, this corresponds to the usual product action of ${G\times H}$ on the Cartesian product ${X\times Y}$, and the corresponding flow is minimal if and only if $X$ and $Y$ are both minimal. 
\medskip

We obtain a partial generalization of this result to the noncommutative setting. We will require the following lemma. 

\begin{lemma}\label{tensor-product-open-proj-dense}
Let ${(\A,G)}$ and ${(\B,H)}$ be strongly minimal C*-flows. If ${e\in(\A\otimes_\mathrm{min}\B)^{**}}$ is a nonzero ${(G\times H)}$-invariant open projection, then $e$ is dense and 
\eqnsp{(\varphi\otimes\psi)^{**}(e)=1}
for all ${\varphi\in S(\A)}$ and ${\psi\in S(\B)}$. 
\end{lemma}

\begin{proof}
Let ${\varphi\in S(\A)}$; then the slice map ${\varphi\otimes\iota:\A\otimes_\mathrm{min}\B\to\B}$ is positive, whence in turn ${(\varphi\otimes\iota)^{**}}$ is positive. This implies that ${(\varphi\otimes\iota)^{**}(e)\in\B^{**}_\lsc}$; moreover, as ${(\varphi\otimes\iota)^{**}}$ is $H$-equivariant, it follows that ${(\varphi\otimes\iota)^{**}(e)}$ is $H$-invariant. Now let ${\psi\in S(\B)}$; then it follows identically that ${(\iota\otimes\psi)^{**}(e)}$ is $G$-invariant. Proposition \ref{strong-minimality-fixed-point-char} therefore implies that there exist ${\lambda_\varphi,\mu_\psi\in\R}$ such that 
\eqnsp{(\varphi\otimes\iota)^{**}(e)=\lambda_\varphi\msp1_{\B^{**}},\qquad\qquad(\iota\otimes\psi)^{**}(e)=\mu_\psi\msp1_{\A^{**}}}
As 
\eqnsp{\lambda_\varphi=\psi^{**}(\lambda_\varphi\msp1_{\B^{**}})=\psi^{**}((\varphi\otimes\iota)^{**}(e))=\varphi^{**}((\iota\otimes\psi)^{**}(e))=\varphi^{**}(\mu_\psi\msp1_{\A^{**}})=\mu_\psi}
for all ${\varphi\in S(\A)}$ and ${\mu\in S(\B)}$, it follows that there exists some ${\lambda\in\R}$ such that ${\lambda_\varphi=\mu_\psi=\lambda}$ for all ${\varphi\in S(\A)}$ and ${\psi\in S(\B)}$. In turn, this implies that 
\eqnsp{(\varphi\otimes\psi)^{**}(\lambda\msp1_{(\A\,\otimes_\mathrm{min}\,\B)^{**}}-e)=0}
for all ${\varphi\in S(\A)}$ and ${\psi\in S(\B)}$, and by taking linear combinations it then follows that 
\eqnsp{\theta^{**}(\lambda\msp1_{(\A\,\otimes_\mathrm{min}\,\B)^{**}}-e)=0}
for all ${\theta\in\A^*\otimes\B^*}$. As $e$ is open and nonzero, there exists a nonzero positive element ${x\in\A_+}$ such that ${x\leq e}$ with ${\norm{x}=1}$. Let ${\eps>0}$; then it follows by \cite[Lemma 2.4]{el77} that there exists a state ${\tau\in S(\A\otimes_\mathrm{min}\B)\cap(\A^*\otimes\B^*)}$ such that ${\tau(x)\geq1-\eps}$. This implies that 
\eqnsp{1-\eps\leq\tau(x)\leq\tau^{**}(e)=\lambda}
and hence by letting ${\eps\to0}$ it follows that ${\lambda\geq1}$. As ${\norm{e}=1}$, evidently ${\lambda\leq1}$, and thus ${\lambda=1}$. Now let ${f\in(\A\otimes_\mathrm{min}\B)^{**}}$ be the closure of $e$ and suppose ${f\neq1_{(\A\,\otimes_\mathrm{min}\,\B)^{**}}}$. Let ${f^\perp=1_{(\A\,\otimes_\mathrm{min}\,\B)^{**}}-f}$; then ${\smash{f^\perp}}$ is a nonzero ${(G\times H)}$-invariant open projection, and thus 
\eqnsp{(\varphi\otimes\psi)^{**}(f^\perp)=1}
for all ${\varphi\in S(\A)}$ and ${\psi\in S(\B)}$. However, this implies that 
\eqnsp{2=(\varphi\otimes\psi)^{**}(e)+(\varphi\otimes\psi)^{**}(f^\perp)\leq(\varphi\otimes\psi)^{**}(f)+(\varphi\otimes\psi)^{**}(f^\perp)=1}
for all ${\varphi\in S(\A)}$ and ${\psi\in S(\B)}$, which is evidently a contradiction; thus $e$ is dense. 
\end{proof}

The following lemma is likely well-known to experts, but we were unable to locate a reference in the literature, so we include a proof for completeness. 

\begin{lemma}\label{type-I-norm-closure-lemma}
Let $\A$ and $\B$ be unital C*-algebras. If $\A$ is type I, then the norm closure of ${\A^*\otimes\B^*}$ in ${(\A\otimes_\mathrm{min}\B)^*}$ contains ${P(\A\otimes_\mathrm{min}\B)}$. 
\end{lemma}

\begin{proof}
Let ${\psi\in P(\A\otimes_\mathrm{min}\B)}$ and let ${(\H_\psi,\pi_\psi,\xi_\psi)}$ be a GNS triple for $\psi$. As $\psi$ is pure, in turn ${\pi_\psi}$ is irreducible; as $\A$ is type I, there are irreducible representations ${\rho:\A\to B(\K)}$ and ${\sigma:\B\to B(\mathcal{L})}$ such that ${\pi_\psi}$ is unitarily equivalent to ${\rho\otimes\sigma}$, and hence there exists a unit vector ${\eta\in\K\otimes\mathcal{L}}$ such that ${\psi=\inner{(\rho\otimes\sigma)(\acdot)\,\eta}{\eta}}$. Expanding $\eta$ as a (possibly infinite) sum of elementary tensors in ${\K\otimes\mathcal{L}}$ shows that $\psi$ is a norm-convergent sum of functionals in ${\A^*\otimes\B^*}$, which proves the claim. 
\end{proof}

\begin{remark}
The converse of the preceding lemma holds as well, but as we will not need this fact, we make no attempt to prove it here (suspecting, as alluded to above, that this is likely folklore). 
\end{remark}

We are now ready to prove Theorem \ref{tensor-product-minimality}. 

\begin{proof}[Proof of Theorem \ref{tensor-product-minimality}]
Suppose that ${(\A\otimes_\mathrm{min}\B,G\times H)}$ is strongly minimal; then ${(\A,G)}$ and ${(\B,H)}$ are strongly minimal by Proposition \ref{strong-minimality-hereditary}. Conversely, suppose that ${(\A,G)}$ and ${(\B,H)}$ are strongly minimal; then Lemma \ref{tensor-product-open-proj-dense} implies that every nonzero ${(G\times H)}$-invariant open projection satisfies the equality ${(\varphi\otimes\psi)^{**}(e)=1}$ for every ${\varphi\in S(\A)}$ and ${\psi\in S(\B)}$. It follows by Lemma \ref{type-I-norm-closure-lemma} that the norm closure of ${\A^*\otimes\B^*}$ in ${(\A\otimes_\mathrm{min}\B)^*}$ contains ${P(\A\otimes_\mathrm{min}\B)}$, and therefore it follows by taking linear combinations and norm limits that ${\tau^{**}(e)=1}$ for every ${\tau\in P(\A\otimes_\mathrm{min}\B)}$. As $e$ is open, this implies that ${e=1_{(\A\,\otimes_\mathrm{min}\,\B)^{**}}}$; the claim then follows from Proposition \ref{strong-minimality-nc-topology-char}. 
\end{proof}

\begin{remark}
It is unclear whether the assumption that $\A$ is type I can be eliminated. If it could be, this would significantly upgrade the result. In particular, this upgraded version of Theorem \ref{tensor-product-minimality} would resolve \cite[Question 7.3]{rordam23}: if ${\B_1\subseteq\A_1}$ and ${\B_2\subseteq\A_2}$ are C*-irreducible inclusions, then it follows by Theorem \ref{c*-irreducibility-char} that ${(\A_1,U(\B_1))}$ and ${(\A_2,U(\B_2))}$ are strongly minimal. The upgrade would then imply that ${(\A_1\otimes_\mathrm{min}\A_2,U(\B_1)\times U(\B_2))}$ is strongly minimal, and as 
\eqnsp{\Span(\{u\otimes v:u\in U(\B_1),\hsp v\in U(\B_2)\})=\B_1\otimes\B_2}
it would then follow by Theorem \ref{c*-irreducibility-char} that ${\B_1\otimes_\mathrm{min}\B_2\subseteq\A_1\otimes_\mathrm{min}\A_2}$ is C*-irreducible. Conversely, it could be shown that the assumption that $\A$ is type I is unnecessary by resolving \cite[Question 7.3]{rordam23} in the affirmative: if ${(\A,G)}$ and ${(\B,H)}$ are strongly minimal C*-flows, then Proposition \ref{strong-minimality-c*-irreducibility-char} implies that ${C_\lambda^*(\Fb_G)\subseteq\A\rtimes_\lambda\Fb_G}$ and ${C_\lambda^*(\Fb_H)\subseteq\B\rtimes_\lambda\Fb_H}$ are C*-irreducible inclusions, and thus 
\eqnsp{C_\lambda^*(\Fb_G)\otimes_\mathrm{min}C_\lambda^*(\Fb_H)\subseteq(\A\rtimes_\lambda\Fb_G)\otimes_\mathrm{min}(\B\rtimes_\lambda\Fb_H)\cong(\A\otimes_\mathrm{min}\B)\rtimes_\lambda(\Fb_G\times\Fb_H)}
is C*-irreducible as well, whence Theorem \ref{c*-irreducibility-char} and Proposition \ref{strong-minimality-hereditary} imply that ${(\A\otimes_\mathrm{min}\B,G\times H)}$ is strongly minimal. 
\end{remark}

We recall that minimal $G$-flows $X$ and $Y$ are said to be \emph{disjoint} when ${X\times Y}$ is a minimal $G$-flow with respect to the diagonal action; equivalently, two (unital) commutative C*-flows are disjoint if their minimal tensor product is (strongly) minimal when equipped with the diagonal action. Our next result shows that if $G$ is C*-simple, then ${(C_\lambda^*(G),G)}$ is disjoint from every strongly minimal C*-flow (including itself). 
\medskip

We will require the following lemma. 

\begin{lemma}\label{inner-crossed-tensor-product}
Let ${(\A,G)}$ be a C*-flow. If the action of $G$ on $\A$ is inner, then ${(\A\otimes_\mathrm{min}C_\lambda^*(G),G)}$ is strongly minimal if and only if ${(\A\rtimes_\lambda G,G)}$ is strongly minimal. 
\end{lemma}

\begin{proof}
Let ${\pi:G\to U(\A)}$ denote the homomorphism implementing the action of $G$ on $\A$; then the linear map ${\varphi:\A\rtimes_\lambda G\to\A\otimes_\mathrm{min}C_\lambda^*(G)}$ specified by letting 
\eqnsp{\varphi(x\lambda_s)=x\msp\pi(s)\otimes\lambda_s}
for every ${x\in\A}$ and ${s\in G}$ is a $G$-equivariant *-isomorphism. As strong minimality is preserved by equivariant *-isomorphisms of C*-flows, the claim follows. 
\end{proof}

\begin{prop}\label{strongly-minimal-tensor-product-char}
If ${(\A,G)}$ is a C*-flow, then ${(\A\otimes_\mathrm{min}C_\lambda^*(G),G)}$ is strongly minimal if and only if $G$ is C*-simple and ${(\A,G)}$ is strongly minimal. 
\end{prop}

\begin{proof}
The forward direction follows from Proposition \ref{strong-minimality-hereditary}. Suppose that $G$ is C*-simple and ${(\A,G)}$ is strongly minimal; then applying Proposition \ref{strongly-minimal-crossed-product-char} twice in succession implies ${((\A\rtimes_\lambda G)\rtimes_\lambda G,G)}$ is also strongly minimal. As the action of $G$ on ${\A\rtimes_\lambda G}$ is inner, it follows from Lemma \ref{inner-crossed-tensor-product} that the product C*-flow ${((\A\rtimes_\lambda G)\otimes_\mathrm{min}C_\lambda^*(G),G)}$ is strongly minimal. As ${\A\otimes_\mathrm{min}C_\lambda^*(G)}$ is a factor of ${(\A\rtimes_\lambda G)\otimes_\mathrm{min}C_\lambda^*(G)}$, Proposition \ref{strong-minimality-hereditary} then yields the claim. 
\end{proof}

We have now gathered both of the pieces which constitute Theorem \ref{crossed-tensor-product-minimality-char}. 

\begin{proof}[Proof of Theorem \ref{crossed-tensor-product-minimality-char}]
The claim follows immediately from the combination of Proposition \ref{strongly-minimal-crossed-product-char} and Proposition \ref{strongly-minimal-tensor-product-char}. 
\end{proof}

\begin{remark}\label{non-type-I-product-minimality}
If $G$ is a C*-simple discrete group, then Proposition \ref{strongly-minimal-tensor-product-char} implies that the product flow ${(C_\lambda^*(G)\otimes_\mathrm{min}\B,G\times H)}$ is strongly minimal for every strongly minimal C*-flow ${(\B,H)}$, which yields a modest extension of Theorem \ref{tensor-product-minimality} beyond the type I case. This can be seen as follows: it is evident that ${(\B,G\times\Fb_H)}$ is strongly minimal, where $G$ acts trivially and ${\Fb_H}$ is equipped with the action induced by its quotient onto $H$. As ${G\times\Fb_H}$ is C*-simple, it follows by Proposition \ref{strongly-minimal-tensor-product-char} that the diagonal flow ${(C_\lambda^*(G\times\Fb_H)\otimes_\mathrm{min}\B,G\times\Fb_H)}$ is strongly minimal. As ${C_\lambda^*(G)\otimes_\mathrm{min}\B}$ is a factor of this C*-flow, Proposition \ref{strong-minimality-hereditary} then implies that ${(C_\lambda^*(G)\otimes_\mathrm{min}\B,G\times\Fb_H)}$ is strongly minimal, and thus in turn ${(C_\lambda^*(G)\otimes_\mathrm{min}\B,G\times H)}$ is strongly minimal. 
\end{remark}

\begin{remark}\label{nc-disjointness-remark}
Proposition \ref{strongly-minimal-tensor-product-char} furnishes an interesting example of a genuinely noncommutative phenomenon: although ${(C_\lambda^*(G)\otimes_\mathrm{min}C_\lambda^*(G),G)}$ is strongly minimal for any C*-simple group $G$, it is never the case that ${(X\times X,G)}$ is minimal for any $G$-flow $X$ where ${\abs{X}\geq2}$, since the diagonal is always a nontrivial $G$-invariant closed subset. Another example of a nontrivial noncommutative strongly minimal C*-flow which is disjoint from every other strongly minimal C*-flow can be found in \cite{avitzour84}. 
\end{remark}

Invoking Proposition \ref{strongly-minimal-crossed-product-char} and unfolding the action as in the proof of Proposition \ref{strongly-minimal-tensor-product-char} yields the following proposition. 

\begin{prop}
Let ${H\leq G}$ be an inclusion of discrete groups. If the inclusion ${C_\lambda^*(H)\subseteq C_\lambda^*(G)}$ is C*-irreducible, then ${C_\lambda^*(\Delta_H^N)\subseteq C_\lambda^*(G\times H^{N-1})}$ is C*-irreducible for all ${N\in\Nb}$, where 
\eqnsp{\Delta_H^N=\{(s,\dots,s)\in G\times H^{N-1}:s\in H\}}
In particular, if $G$ is C*-simple, then ${C_\lambda^*(\Delta_G^N)\subseteq C_\lambda^*(G^N)}$ is C*-irreducible for all ${N\in\Nb}$. 
\end{prop}

\section{Examples}\label{example-section}

In this section we consider various examples of strongly minimal C*-flows. The first examples in this section are worked out by hand, while the remainder of the section is devoted to obtaining a characterization of strong minimality for arbitrary subhomogeneous C*-flows and then examining such examples using our accumulated machinery. 
\medskip

Our first example is a noncommutative analogue of the example considered in Remark \ref{ergodic-non-minimal-remark}. 

\begin{example}
Let $\H$ be a Hilbert space and $G$ be a discrete group. Suppose that ${(B(\H),G)}$ is a C*-flow. As every automorphism of ${B(\H)}$ fixes the set of all minimal projections in ${B(\H)}$, in turn every such automorphism fixes ${K(\H)}$, and therefore ${K(\H)}$ is $G$-invariant. In particular, ${(B(\H),G)}$ is never minimal when $\H$ is infinite dimensional, and hence never strongly minimal. On the other hand, when $\H$ is finite dimensional ${(B(\H),G)}$ is strongly minimal if and only if it is ergodic as a W*-flow (see Example \ref{motivating-example-ii}). 
\end{example}

The phenomenon exhibited in the above example for the case where $\H$ is infinite dimensional is a simple obstruction to a C*-algebra having any strongly minimal actions. In general, a C*-algebra having no strongly minimal actions is equivalent to it having no minimal actions: this can be seen from the fact that the inner automorphism group fixing a left ideal implies it is a two-sided ideal. 
\medskip

The next example gives a family of strongly minimal actions that work for all rotation algebras (in particular, these actions are not just the conjugation action when the algebras are simple).

\begin{example}\label{ExRot}
Let ${\theta\in\R}$ and ${\omega=e^{2\pi i\theta}}$. We recall that the \emph{rotation algebra} ${\A_\theta}$ with parameter $\theta$ is the universal C*-algebra generated by two unitaries $u$ and $v$ subject to the relation 
\eqnsp{uv=\omega vu}
Let ${\alpha_\theta:\Zb\to\Aut(C(\Tb))}$ be the group homomorphism defined by 
\eqnsp{\alpha_\theta(n)(f)(z)=f(\omega^{-n}z)}
It is a standard fact that ${\A_\theta\cong C(\Tb)\rtimes_{\alpha_\theta,\msp\lambda}\Zb}$, where $u$ is identified with ${\mathrm{id}_\Tb\in C(\Tb)\subseteq C(\Tb)\rtimes_{\alpha_\theta,\msp\lambda}\Zb}$ and $v$ is identified with ${\lambda_1\in C_\lambda^*(\Zb)\subseteq C(\Tb)\rtimes_{\alpha_\theta,\msp\lambda}\Zb}$. Define an automorphism ${\beta\in\Aut(\A_\theta)}$ by 
\eqnsp{\beta(u)=v^*,\qquad\qquad\beta(v)=u}
This is a well-defined *-homomorphism since ${v^*u=\omega uv^*}$, and it is an automorphism (of period $4$) because ${\beta^4=\mathrm{id}_{\A_\theta}}$. We observe that ${\beta(C(\Tb))=C_\lambda^*(\Zb)}$ and ${\beta(C_\lambda^*(\Zb))=C(\Tb)}$ under the previously mentioned isomorphism of ${\A_\theta}$ and ${C(\Tb)\rtimes_{\alpha_\theta,\msp\lambda}\Zb}$. Fix some ${\xi\in\Tb}$ which is not a root of unity; then there is a natural automorphism ${\gamma_\xi\in\Aut(\A_\theta)}$ given by 
\eqnsp{\gamma_\xi(u)=\xi u,\qquad\qquad\gamma_\xi(v)=v}
When viewed inside the crossed product ${C(\Tb)\rtimes_{\alpha_\theta,\msp\lambda}\Zb}$, we see that ${\gamma_\xi}$ acts trivially on ${C_\lambda^*(\Zb)}$ and that the action on ${C(\Tb)}$ arises from the action ${z\to\xi z}$ on ${\Tb}$. These two automorphisms induce an action of ${G=\Zb/4\Zb*\Zb}$ on ${\A_\theta}$, where ${\Zb/4\Zb\subseteq\Zb/4\Zb*\Zb}$ acts via $\beta$ and ${\Zb\subseteq\Zb/4\Zb*\Zb}$ acts via ${\gamma_\xi}$. 
\medskip

\noindent We will now show that ${(\A_\theta,G)}$ is a strongly minimal C*-flow. Let $\mathcal I$ be a nonzero closed $G$-invariant left ideal in ${\mathcal A_\theta}$. Note that if ${x\in A_\theta\cong C(\mathbb T)\rtimes_{\alpha_\theta, \lambda} \mathbb Z}$, then $x$ has a unique Fourier series ${x\sim \sum_{n=-\infty}^\infty x_n \lambda_n}$, where ${x_n\in C(\mathbb T)}$. It is easy to see that ${\gamma_\xi(x)\sim \sum_{n=-\infty}^\infty \gamma_{\xi}(x_n)\lambda_n}$. It therefore follows by Birkhoff's ergodic theorem that if ${P_nx=\frac{1}{n}\sum_{k=0}^{n-1} \gamma_{\xi}^n(x)}$, then ${(P_nx)_{n=1}^\infty}$ converges in norm to an element in ${C(\mathbb T)\rtimes_{\alpha_\theta, \lambda}\mathbb Z}$ with Fourier series ${\sum_{n=-\infty}^\infty m(x_n)\lambda_n}$, where $m$ is the Lebesgue measure on $\mathbb T$. Note that this element actually belongs to ${C_\lambda^* (\mathbb Z)}$ because all of the coefficients are complex numbers and ${\mathbb Z}$ is amenable. Now, if ${x\in \mathcal I}$ is nonzero and positive, then ${P_nx\in \mathcal I}$ for all $n$, so it follows that ${C_\lambda^* (\mathbb Z)\cap \mathcal I\neq 0}$ as ${m(x_0)\neq 0}$ as $m$ is faithful and ${x_0}$ is nonzero and positive (this is because the canonical conditional expectation for the reduced crossed product is faithful). As $\beta$ swaps ${C_\lambda^*( \mathbb Z)}$ and ${C(\mathbb T)}$ it follows that ${C(\mathbb T)\cap \mathcal I\neq 0}$. As ${C(\mathbb T)\cap \mathcal I}$ is a closed left ideal and is $\mathbb Z$-invariant, it follows that ${C(\mathbb T)\cap \mathcal I=C(\mathbb T)}$ as the $\mathbb Z$-action on $\mathbb T$ is minimal. It follows that ${1\in \mathcal I}$, so ${\mathcal I=\mathcal A_\theta}$. As $G$ is C*-simple and ${(\mathcal A_\theta, G)}$ is a strongly minimal C*-flow, Proposition \ref{strongly-minimal-crossed-product-char} combined with Theorem \ref{c*-irreducibility-char} shows that ${C_\lambda^* (G)\subseteq \mathcal A_\theta\rtimes _\lambda G}$ is a C*-irreducible inclusion.
\end{example}

We will now investigate strong minimality for C*-flows wherein the underlying C*-algebra is type I or is somehow related to type I C*-algebras. 
\medskip

Some basic knowledge about the spectrum of a general C*-algebra will be needed in the sequel; the unfamiliar reader should consult \cite[Chapter 4]{pedersen18} for more details. Suppose that ${(\mathcal A, G)}$ is a C*-flow. When this is the case, there are natural action of $G$ on both the spectrum ${\widehat{\mathcal A}}$ and the primitive ideal space ${\Prim(\A)}$ of $\mathcal A$ by homeomorphisms. The action on ${\widehat{\mathcal A}}$ is defined by letting ${s[\pi]=[s\pi]}$ for ${s\in G}$ and $\pi$ an irreducible representation of $\mathcal A$, where ${s\pi}$ is the irreducible representation which is defined by ${s\pi(x)=\pi(s^{-1}x)}$ for ${x\in \mathcal A}$. The action on ${\Prim(\mathcal A)}$ is obtained by simply applying a given automorphism from $G$ to a given primitive ideal in $\mathcal A$. It is easy to verify that the natural map ${\widehat{\mathcal A}\to \Prim(\mathcal A)}$ and the natural order preserving bijections between the set of closed two-sided ideals of $\mathcal A$ and the open subsets of either ${\widehat{\mathcal A}}$ or ${\Prim(\mathcal A)}$ are all $G$-equivariant. Combining these facts with Proposition \ref{strong-minimality-hereditary-minimality-char} easily yields the following.

  \begin{prop}\label{SpecMin}
Let ${(\A,G)}$ be a C*-flow. The following are equivalent: 
\begin{enumerate}
\item ${(\A,G)}$ is strongly minimal. 
\item ${(\smash{\widehat{\B}},G)}$ is minimal for every factor ${\B\subseteq\A}$. 
\item ${(\Prim(\B),G)}$ is minimal for every factor ${\B\subseteq\A}$. 
\end{enumerate}
  \end{prop}

A C*-algebra $\mathcal A$ is \emph{homogeneous} (respectively, \emph{subhomogeneous}) \emph{of degree $n$} for some ${n\in\Nb}$ if every irreducible representation of $\mathcal A$ has dimension equal (respectively, less than or equal) to $n$, and simply \emph{homogeneous} (respectively, \emph{subhomogeneous}) if it is homogeneous (respectively, subhomogeneous) of degree $n$ for some ${n\in\Nb}$. Clearly every homogeneous C*-algebra is subhomogeneous, and both are examples of type I C*-algebras. Subhomogeneous C*-algebras are simultaneously generalizations of commutative C*-algebras and matrix algebras.  Accordingly, the prototypical example of a unital homogeneous C*-algebra of degree $n$ is ${C(X,M_n)\cong C(X)\otimes M_n}$, where $X$ is a compact Hausdorff space and ${n\in\Nb}$. In fact, every subhomogeneous C*-algebra is isomorphic to a C*-subalgebra of a C*-algebra of the aforementioned type. The next corollary shows that we may restrict our attention to homogeneous C*-algebras when studying strong minimality. 

  \begin{cor}\label{HomogSubalg}
      Let ${(\mathcal A, G)}$ be a strongly minimal C*-flow and ${\B\subseteq\A}$ be a factor. If $\B$ admits a nonzero finite-dimensional representation, then $\mathcal B$ is homogeneous.
  \end{cor}

  \begin{proof}
      If $\mathcal B$ has a nonzero finite-dimensional representation then it has a nonzero finite-dimensional irreducible representation. Let ${n\in\Nb}$ be the least positive integer ${m\in\Nb}$ for which there exists a finite-dimensional irreducible representation of $\mathcal B$ of dimension $m$. If ${F\subseteq \widehat{\mathcal B}}$ denotes the collection of equivalence classes of irreducible representations of dimension less than or equal to $n$, then $F$ is closed (by \cite[Proposition 3.6.3]{dixmier77}) and $G$-invariant (as $\pi$ and ${s\pi}$ have the same dimension for any representation $\pi$ and ${s\in G}$). By Proposition \ref{SpecMin}, the action of $G$ on ${\widehat{\mathcal B}}$ is minimal, so as ${F\neq \emptyset}$ it follows that ${F=\widehat{\mathcal B}}$. As $n$ is the least integer for which there exists an irreducible representation of $\mathcal B$ of said dimension, it follows that $\mathcal B$ is homogeneous of degree $n$.
  \end{proof}

  In light of Corollary \ref{HomogSubalg}, it makes sense to try and find a characterization of strong minimality of C*-flows where the underlying C*-algebras are homogeneous. In this section, we obtain a simple characterization of strong minimality in this situation. This result will require a small amount of preparation.
  \medskip

  As mentioned above, every unital subhomogeneous C*-algebra (and hence every unital homogeneous C*-algebra) is isomorphic to a C*-subalgebra of ${C(X,M_n)}$ for some compact Hausdorff space $X$ and ${n\in\Nb}$. As strong minimality is a condition on closed $G$-invariant left ideals, it will be helpful to have a result that characterizes the structure of the closed left ideals for C*-subalgebras of ${C(X,M_n)}$. This is exactly what the next result entails.

  \begin{lemma}\label{LeftHomog}
  Let $\mathcal A$ be a C*-algebra, $X$ be a locally compact Hausdorff space, and ${\mathcal B\subseteq C_0(X,\A)}$ be a C*-subalgebra. If ${\mathcal I\subseteq\B}$ is a closed left ideal, then there exists a collection ${(\mathcal I_x)_{x\in X}}$ of closed left ideals in $\mathcal A$ such that
 \eqnsp{\mathcal I=\{f\in \mathcal B:f(x)\in \mathcal I_x\ \text{for all }x\in X\}}
and which is minimal in the sense that if ${(\mathcal I'_x)_{x\in X}}$ is another collection of closed left ideals that satisfies the above equality, then ${\mathcal I_x\subseteq \mathcal I'_x}$ for every ${x\in X}$.
  \end{lemma}

  \begin{proof}
      For ${x\in X}$, let ${\mathcal J_x=\{f(x):f\in \mathcal I\}}$ and ${\mathcal B_x=\{f(x):f\in \mathcal B\}}$. It is clear for each ${x\in X}$ that ${\mathcal B_x}$ is a C*-subalgebra of $\mathcal A$ and it follows from the Cohen-Hewitt factorization theorem that ${\mathcal J_x}$ is a closed left ideal in ${\mathcal B_x}$. Let ${\mathcal I_x}$ be the least closed left ideal in $\mathcal A$ with ${\mathcal J_x=\mathcal I_x\cap \mathcal B_x}$ (so if ${p\in \mathcal B_x^{**}\subseteq \mathcal A^{**}}$ is the support projection for ${\mathcal J_x}$, then ${\mathcal I_x=\mathcal A^{**}p\cap \mathcal A}$). Define ${\mathcal J\subseteq \mathcal B}$ to be equal to
      \eqnsp{\{f\in \mathcal B:f(x)\in \mathcal I_x\ \text{for all } x\in X\}}

      \noindent It is clear from the definition of the family ${(\mathcal I_x)_{x\in X}}$ that ${\mathcal I\subseteq \mathcal J}$ and also that $\mathcal J$ is a closed left ideal in $\mathcal B$. Thus, to prove that ${\mathcal I=\mathcal J}$, it suffices to prove that every pure state which vanishes on $\mathcal I$ also vanishes on $\mathcal J$. So let $\phi$ be a pure state on $\mathcal B$ which vanishes on $\mathcal I$. As $\phi$ can be extended to a pure state on ${C_0(X,\mathcal A)}$, it follows from \cite[Chapter IV, Theorem 4.14]{takesaki02} that there is a pure state $\psi$ on $\mathcal A$ and an ${x\in X}$ with ${\phi(f)=\psi(f(x))}$ for every ${f\in \mathcal B}$. As $\phi$ vanishes on $\mathcal I$, it follows that ${\psi(\mathcal J_x)=0}$. Therefore for ${f\in \mathcal J}$ it must be the case that ${f(x)\in \mathcal I_x\cap \mathcal B_x=\mathcal J_x}$ and consequently that ${\phi(f)=\psi(f(x))=0}$. This proves that ${\mathcal I=\mathcal J}$. The family ${(\mathcal I_x)_{x\in X}}$ being minimal is an immediate consequence of its construction.
  \end{proof}

  It will be useful to have a concrete description of a homogeneous C*-algebra for the forthcoming proofs. If $\mathcal A$ is a unital C*-algebra which is homogeneous of degree $n$, then we let $X_\mathcal A$ denote the set of all nonzero *-homomorphisms $\mathcal A\to M_n$ equipped with the product topology. It is easy to verify that ${X_\mathcal A}$ is compact. As $\mathcal A$ is homogeneous of degree $n$ it is easy to see that ${X_\mathcal A}$ consists of irreducible representations of $\mathcal A$ and that every irreducible representation of $\mathcal A$ is unitarily equivalent to an element in ${X_\mathcal A}$. In particular, letting ${\Phi_\mathcal A:\mathcal A\to C(X_\mathcal A, M_n)}$ be defined by
  \eqnsp{(\Phi_\mathcal A(a))(\pi)=\pi(a),\qquad\pi\in X_\mathcal A,\ a\in \mathcal A}

  \noindent it follows that ${\Phi_\mathcal A}$ is an injective *-homomorphism. Thus, ${\Phi_\mathcal A(\mathcal A)\subseteq C(X_\mathcal A, M_n)}$ gives a concrete representation of $\mathcal A$ as a C*-subalgebra of one of the model spaces for homogeneous C*-algebras. There is a natural action of ${U(M_n)}$ on ${X_\mathcal A}$ defined by
  \eqnsp{(u\cdot \pi)(a)=u\pi(x)u^*,\ u\in U(M_n),\qquad\pi\in X_\mathcal A,\ a\in \mathcal A}
It is obvious from the definitions that every ${f\in \Phi_\mathcal A(\mathcal A)}$ is ${ U(M_n)}$-equivariant. Slightly less obvious, but equally important for the sequel, is that ${\Phi_\mathcal A(\mathcal A)}$ consists of precisely the functions in ${C(X_\mathcal A, M_n)}$ which are ${ U(M_n)}$-equivariant (a proof of this assertion can be found in the proof of \cite[Theorem 5.2]{niemiec15}). In light of the above, the C*-algebra ${\Phi_\mathcal A(\mathcal A)\subseteq C(X_\mathcal A, M_n)}$ consisting exactly of the continuous ${ U(M_n)}$-equivariant functions ${X_\mathcal A\to M_n}$ will be referred to as \textit{the canonical representation of $\mathcal A$}.
  \medskip

  If ${(\mathcal A, G)}$ is a C*-flow with $\mathcal A$ homogeneous of degree $n$, then the action of $G$ on $\mathcal A$ becomes extremely simple when viewed inside the canonical representation of $\mathcal A$. Indeed, there is a natural action of $G$ on ${X_\mathcal A}$ (defined in the same way as the action on ${\widehat{\mathcal A}}$) and this in turn defines an action on ${C(X_\mathcal A, M_n)}$. As can easily be checked, the ${ U(M_n)}$-action on ${X_\mathcal A}$ commutes with the $G$-action on ${X_\mathcal A}$, so it follows that ${\Phi_\mathcal A}$ is $G$-equivariant. Thus, if ${(\mathcal A, G)}$ is a C*-dynamical system with $\mathcal A$ homogeneous of degree $n$ then there is an action of $G$ on ${X_\mathcal A}$ which commutes with the ${ U(M_n)}$-action so that
  \eqnsp{(sf)(\pi)=f(s^{-1}\pi),\qquad s\in G,\ f\in \Phi_\mathcal A(\mathcal A),\ \pi\in X_\mathcal A}

  The discussion above leads naturally to the following upgrade of Lemma \ref{LeftHomog} for homogeneous C*-algebras.

  \begin{lemma}\label{HomogCanStruc}
      Let ${(\mathcal A, G)}$ be a C*-flow with $\mathcal A$ homogeneous of degree $n$. If ${\mathcal I\subseteq\A}$ is a closed left ideal, then there is a collection ${(p_\pi)_{\pi\in X_\mathcal A}}$ of projections in ${M_n}$ with
      \eqnsp{\mathcal I=\{a\in \mathcal A:\pi(a)\in M_np_\pi\ \text{for every } \pi\in X_\mathcal A\}}

      \noindent and such that the map ${\pi\mapsto p_\pi}$ is ${ U(M_n)}$-equivariant. If $\mathcal I$ is also $G$-invariant, then the map ${\pi\mapsto p_\pi}$ is additionally $G$-equivariant.
  \end{lemma}

  \begin{proof}
 We recall that every left ideal of ${M_n}$ is of the form ${M_np}$ for a projection $p$ in ${M_n}$. Viewing $\mathcal A$ in its canonical representation and applying Lemma \ref{LeftHomog} yields a collection ${(p_\pi)_{\pi\in X_\mathcal A}}$ of projections in ${M_n}$ with 
      \eqnsp{\mathcal I=\{a\in \mathcal A:\pi(a)\in M_np_\pi\ \text{for every } \pi\in X_\mathcal A\}}

      \noindent and such that if ${(p'_\pi)_{\pi\in X_\mathcal A}}$ is another collection of projections satisfying the above equation then ${p_\pi\leq p'_\pi}$. Now if ${a\in \mathcal A}$, ${u\in  U(M_n)}$, and ${\pi \in X_\mathcal A}$, then ${u\pi(a)u^*=(u\cdot \pi)(a)\in M_np_{u\cdot\pi}}$ if and only if ${\pi(a)\in  M_n(u^*p_{u\cdot \pi}u)}$. For a fixed ${u\in  U(M_n)}$ it follows that 
      \eqnsp{\mathcal I=\{a\in \mathcal A:\pi(a)\in M_n(u^*p_{u\cdot \pi}u)\ \text{for every }\pi\in X_\mathcal A\}}

      \noindent so by minimality of the family ${(p_\pi)_{\pi\in X_\mathcal A}}$ it must be that ${p_\pi\leq u^*p_{u\cdot \pi}u}$ and hence ${up_\pi u^*\leq p_{u\cdot \pi}}$ for every ${\pi\in X_\mathcal A}$. Replacing $u$ with ${u^*}$ and $\pi$ with ${u\cdot \pi}$ yields ${p_{u\cdot \pi}\leq up_\pi u^*}$ so ${p_{u\cdot\pi}=up_\pi u^*}$. Thus the map ${\pi\mapsto p_\pi}$ is ${ U(M_n)}$-equivariant. If $\mathcal I$ is $G$-invariant, then the same technique shows ${\pi\mapsto p_\pi}$ is $G$-equivariant.
  \end{proof}

  The next lemma is the last piece that is needed for the aforementioned characterization of strong minimality of actions on homogeneous C*-algebras.

  \begin{lemma}\label{OpCloCon}
      Let $X$ be a locally compact Hausdorff space, ${\mathcal B\subseteq C_0(X,\mathcal M_n)}$ be a C*-subalgebra, ${\mathcal I\subseteq\B}$ be a closed left ideal, and ${(p_x)_{x\in X}}$ the projections for $\mathcal I$ in ${M_n}$ from Lemma \ref{LeftHomog}. If for ${k=0,1,\dots, n}$ we let 
      \eqnsp{U_k=\{x\in X:\text{the rank of } p_x\ \text{is at least } k\}}

      \noindent and
      \eqnsp{F_k=\{x\in X:\text{the rank of } p_x\ \text{is equal to } k\}}

      \noindent then the following hold:
      
      \begin{enumerate}
          \item ${U_k}$ is open in $X$;
          \item ${F_k}$ is closed in $U_k$; and 
          \item the map ${x\mapsto p_x}$ is continuous on $F_k$.
      \end{enumerate}
  \end{lemma}

  \begin{proof}
      It will first be shown that ${U_k}$ is open. Fixing ${z\in U_k}$, ${(M_n p_x)_{x\in X}}$ being as in Lemma \ref{LeftHomog} makes it easy to see that there is a positive ${f\in \mathcal I}$ with ${f(z)=p_z}$. As $f$ is continuous and the rank function ${M_n\to \{0,1,\dots, n\}}$ is lower semicontinuous, it follows that there is a neighbourhood $W$ of $z$ in $X$ in which the rank of ${f(w)}$ is at least $k$ for ${w\in W}$. As ${f(w)p_w=f(w)}$ it must be the case that ${p_w}$ has rank at least $k$ for ${w\in W}$. It follows that ${U_k}$ is open. That ${F_k}$ is closed in ${U_k}$ is now immediate as ${F_k=U_k\cap (X\setminus U_{k+1})}$ (where ${U_{k+1}=\emptyset}$ when ${k=n}$). Showing that ${x\mapsto p_x}$ is continuous on ${F_k}$ can be done locally: fix ${z\in F_k}$ and again select a positive ${f\in \mathcal I}$ so that ${f(z)=p_z}$ along with a neighbourhood $W$ of $z$ so that the rank of ${f(w)}$ is at least $k$ when ${w\in W}$. It is well known that the function ${P:M_n\to M_n}$ which maps a matrix to the projection onto its range is continuous on the set of matrices of rank equal to $k$ so it follows that ${P\circ f}$ is continuous on ${W\cap F_k}$. But $f$ being positive implies that ${p_wf(w)=f(w)}$, so combining this with the fact that ${p_w}$ has rank $k$ for ${w\in W\cap F_k}$ it is easily seen that ${P(f(w))=p_w}$ for ${w\in W\cap F_k}$. Thus ${x\mapsto p_x}$ is continuous on the neighbourhood ${W\cap F_k}$ of $z$ in ${F_k}$. Hence ${x\mapsto p_x}$ is continuous on ${F_k}$.
  \end{proof}

  Before diving into the proof of Theorem \ref{subhomogeneous-minimality}, it is important to mention the following: if $\mathcal A$ is a unital homogeneous C*-algebra, then ${\widehat{\mathcal A}}$ is a compact Hausdorff space and the topology is the quotient topology from the action ${ U(M_n)\curvearrowright X_\mathcal A}$ (see \cite[Proposition 3.6.4]{dixmier77}). In particular, if ${(\mathcal A, G)}$ is a C*-flow with $\mathcal A$ homogeneous, then ${\widehat{\mathcal A}}$ is a $G$-flow. 

  \begin{proof}[Proof of Theorem \ref{subhomogeneous-minimality}]
    If ${(\mathcal A, G)}$ is strongly minimal, then Condition (1) holds as a consequence of Corollary \ref{HomogSubalg}, Condition (2) holds as a consequence of Proposition \ref{SpecMin}, and Condition (3) follows from the definition of strong minimality.
    \medskip

    Now suppose that $\mathcal A$ is homogeneous of degree $n$ and the action of $G$ on ${\widehat{\mathcal A}}$ is minimal, but that ${(\mathcal A, G)}$ is not strongly minimal. It follows from the comments preceding this proof that the action of $G$ on ${\widehat{\mathcal A}}$ is minimal if and only if the action of ${ U(M_n)\times G}$ on ${X_\mathcal A}$ is minimal (this being a valid action because the actions commute). As ${(\mathcal A, G)}$ is not strongly minimal, there is a nontrivial closed left ideal $\mathcal I$ of $\mathcal A$ which is $G$-invariant. Applying Lemma \ref{HomogCanStruc} gives a family ${(p_\pi)_{\pi\in X_\mathcal A}}$ of projections in ${M_n}$ such that the map ${\pi\mapsto p_\pi}$ is both ${ U(M_n)}$-equivariant and $G$-invariant. For ${k=0,1,\dots, n}$ let ${U_k}$ and ${F_k}$ be as in the statement of Lemma \ref{OpCloCon}. An easy computation shows that both ${U_k}$ and ${F_k}$ are ${ U(M_n)\times G}$-invariant (as ${\pi\mapsto p_\pi}$ is both ${ U(M_n)}$-equivariant and $G$-invariant), so each ${U_k}$ is either equal to $\emptyset$ or ${X_\mathcal A}$, since ${U_k}$ is open. Note that there is some ${0<k<n}$ with ${F_k\neq \emptyset}$ (otherwise $\mathcal I$ is trivial), so for this choice of $k$, ${U_k\neq \emptyset}$ and thus ${U_k=X_\mathcal A}$, and now as ${F_k}$ is a closed ${ U(M_n)\times G}$-invariant subset of ${U_k=X_\mathcal A}$ it follows that ${F_k=X_\mathcal A}$. As such, ${\pi\mapsto p_\pi}$ is continuous by the third condition in Lemma \ref{OpCloCon} and so the ${ U(M_n)}$-equivariant ${p\in C(X_\mathcal A, M_n)}$ given by ${p(\pi)=p_\pi}$ is a nontrivial $G$-invariant projection in ${\Phi_\mathcal A(\mathcal A)\cong \mathcal A}$ (which has constant rank equal to $k$).
  \end{proof}

  \begin{remark}\label{subhomogeneous-remark}
      It is clear from the proof of Theorem \ref{subhomogeneous-minimality} that Condition (3) in Theorem \ref{subhomogeneous-minimality} can be replaced with the weaker condition that there is no nontrivial $G$-invariant projection in ${C(X_\mathcal A, M_n)}$ which is ${ U(M_n)}$-equivariant and has constant rank.
  \end{remark}

  A particular situation to highlight for applications of Theorem \ref{subhomogeneous-minimality} is when ${\mathcal A=C(X,M_n)}$, where $X$ is a compact Hausdorff space and ${n\in\Nb}$. Note that any automorphism of $\alpha$ of ${C(X, M_n)}$ fixes the centre ${Z(C(X,M_n))=C(X)}$, so gives rise to a natural homeomorphism of $X$: this implies that if ${(C(X,M_n), G)}$ is a C*-dynamical system then there is a naturally induced action of $G$ on $X$.

  \begin{cor}\label{MinHomogSimp}
      Let $X$ be a compact Hausdorff space and ${n\in\Nb}$. If ${(C(X,M_n), G)}$ is a C*-flow, then ${(C(X,M_n), G)}$ is strongly minimal if and only if the following conditions hold:

      \begin{enumerate}
          \item the induced action of $G$ on $X$ is minimal; and
          \item there is no nontrivial $G$-invariant projection in ${C(X,M_n)}$ of constant rank.
      \end{enumerate}
  \end{cor}

  \begin{proof}
      It is clear that if ${(C(X,M_n), G)}$ is strongly minimal then the action of $G$ on $X$ is minimal and that there are no nontrivial $G$-invariant projections in ${C(X,M_n)}$. Conversely, assume that Conditions (1) and (2) hold. Condition (1) in Theorem \ref{subhomogeneous-minimality} holds because ${C(X,M_n)}$ is homogeneous of degree $n$. As the spectrum of ${C(X,M_n)}$ can be naturally identified with $X$, Condition (1) obviously implies Condition (2) of Theorem \ref{subhomogeneous-minimality}. Condition (3) is an easy consequence of Remark \ref{subhomogeneous-remark}: if ${\mathcal A=C(X,M_n)}$, then it is easy to verify that ${X_\mathcal A\cong X\times  U(M_n)/\mathbb T}$ and that this turns ${\Phi_\mathcal A:\mathcal A\to C(X_\mathcal A, M_n)}$ into the map ${(\Phi_\mathcal A(f))(x,u\mathbb T)=uf(x)u^*}$, and this map obviously preserves elements of constant rank.
  \end{proof}

  Corollary \ref{MinHomogSimp} gives a method for obtaining disjoint C*-flows in the sense described in Remark \ref{nc-disjointness-remark} with one C*-flow commutative and one C*-flow noncommutative as ${C(X)\otimes M_n\cong C(X, M_n)}$. A classical theorem of topological dynamics is that minimal proximal flows and minimal distal flows are disjoint (see \cite[Corollary II.1.3.]{glasner76}). In the noncommutative setting, there do not seem to be good notions of proximality and distality, so it is difficult to come up with a noncommutative generalization of this result. If one interprets ${M_n}$ as a noncommutative finite set and notes that any action on a finite set is distal, then the next example gives a small generalization of this phenomenon to the setting where the distal action is on a noncommutative C*-algebra.

  \begin{example}\label{proximal-example}
      Let $G$ be a discrete group. The universal minimal proximal flow ${\partial_p G}$ is the minimal proximal flow which has every other minimal proximal flow as a factor, and is unique up to an isomorphism of $G$-flows (in particular, $\partial_F G$ is a factor of $\partial_p G$). Let ${\alpha:G\to \Aut(M_n)}$ be any action of $G$ on ${M_n}$ for which ${(M_n, G)}$ is strongly minimal. If ${(C(\partial_p G, M_n), G)}$ is the C*-flow with action given by 
      \eqnsp{(sf)(x)=\alpha(s)(f(s^{-1}x))}
      
      \noindent then ${(C(\partial_p G, M_n), G)}$ is strongly minimal. Indeed, the action of $G$ on ${\partial_p G}$ is minimal, so by Corollary \ref{MinHomogSimp} it suffices to show that there are no $G$-invariant elements in ${C(\partial_p G, M_n)}$ which are not scalar multiples of the identity. Suppose therefore that ${f\in C(\partial_p G, M_n)}$ is $G$-invariant. Observe that if ${x,y\in \partial_pG}$, then there is a net ${(s_i)}$ in $G$ with ${\alpha(s_i)\to \text{id}_{M_n}}$ in point-norm and ${s_i x\to y}$. Indeed, $G$ acts minimally and distally on the compact group ${\overline{\alpha(G)}}$, so by \cite[Corollary II.1.3]{glasner76} it follows that there is a net ${(s_i)}$ in $G$ with ${(\alpha(s_i), s_ix)=s_i(\text{id}_{M_n}, x)\to (\text{id}_{M_n}, y)}$. This implies
      \eqnsp{f(x)=\lim_i (s_i^{-1}f)(x)=\lim_i\alpha(s_i)^{-1}(f(s_ix))=f(y)}

      \noindent It follows that $f$ is constant. Fix ${x\in \partial_p G}$. For any ${s\in G}$ it is immediate that ${f(x)=\alpha(s)(f(x))}$, so ${f(x)}$ is fixed by ${\alpha(G)}$, and thus by strong minimality of ${(M_n, G)}$ it follows that ${f(x)}$ is a scalar multiple of the identity in ${M_n}$. Therefore, $f$ is a scalar multiple of the identity in ${C(\partial_p G, M_n)}$ and thus ${(C(\partial_p G, M_n), G)}$ is strongly minimal as claimed. When $G$ is discrete and C*-simple, Proposition \ref{strongly-minimal-crossed-product-char} combined with Theorem \ref{c*-irreducibility-char} shows that the inclusion
      \eqnsp{C_\lambda^* (G)\subseteq C(\partial_p G, M_n)\rtimes_\lambda G}

      \noindent is C*-irreducible. By using Proposition \ref{strong-minimality-hereditary}, ${\partial_p G}$ can be replaced by any minimal proximal $G$-flow. This yields many examples of C*-irreducible inclusions: for instance, when ${G=\mathbb F_2}$ the Gromov boundary ${\partial \mathbb F_2}$ can be used, and when ${G=PSL(n,\mathbb Z)}$ for ${n\geq 3}$, then ${\mathbb R\mathbb P^{n-1}}$ or ${PSL(n,\mathbb R)/Q}$ can be used, where $Q$ is the subgroup of ${PSL(n, \mathbb R)}$ consisting of upper triangular matrices.
  \end{example}

  \begin{example}
We reuse the notation from Example \ref{ExRot} and further assume that $\theta$ is rational. In this example it will be shown that ${(\mathcal A_\theta, G)}$ is strongly minimal as an application of Theorem \ref{subhomogeneous-minimality}. Firstly, as $\theta$ is rational, it is well known that ${\mathcal A_\theta}$ is homogeneous of degree $N$, where $N$ is the least positive integer $n$ with ${\omega^n=1}$. Next it will be shown that ${\mathcal A_\theta^G=\mathbb C}$. Indeed, let ${x\in \mathcal A_\theta^G}$, and let ${x\sim \sum_{n=-\infty }^\infty x_n\lambda_n}$ denote the Fourier series of $x$. As
      \eqnsp{x=\gamma_\xi(x)\sim\sum_{n=-\infty}^\infty \gamma_\xi(x_n)\lambda_n}

      \noindent it follows that ${x_n=\gamma_\xi(x_n)}$, so as the $\mathbb Z$ action of ${\gamma_\xi}$ on ${C(\mathbb T)}$ is minimal it follows that ${x_n\in \mathbb C}$. As $\mathbb Z$ is amenable, this implies that ${x\in C_\lambda^* (\mathbb Z)}$. As ${\beta(x)=x}$ it follows that ${x\in C_\lambda^*(\mathbb Z)\cap C(\mathbb T)=\mathbb C}$. Now it will be shown that ${G\curvearrowright \widehat{\mathcal A_\theta}}$ is minimal. For the most part this is an exercise in computing the spectrum of ${\mathcal A_\theta}$. By the Dauns-Hofmann theorem, ${C(\widehat{A_\theta})\cong Z(\mathcal A_\theta)}$ (as ${\mathcal A_\theta}$ is type I, ${\Prim(\mathcal A_\theta)=\widehat{\mathcal A_\theta}}$), so to determine the spectrum of ${\mathcal A_\theta}$ it suffices to compute the centre of ${\mathcal A_\theta}$ (noting that the isomorphism provided by the Dauns-Hofmann theorem is also equivariant with respect to automorphisms). This can be done by noting that ${Z(\mathcal A_\theta)=C^*(u^N, v^N)\cong C(\mathbb T\times \mathbb T)}$, where the isomorphism is determined by mapping ${u^N}$ to ${\text{id}_\mathbb T\otimes 1}$ and mapping ${v^N}$ to ${1\otimes \text{id}_\mathbb T}$. This shows that the action of $G$ on ${\widehat{\mathcal A_\theta}}$ is transported to ${\mathbb T\times \mathbb T}$ in the following way: $\beta$ maps ${(z,z')}$ to ${(\overline{z'}, z)}$ and ${\gamma_\xi}$ maps ${(z,z')}$ to ${(\xi^{-N}z, z')}$. As ${\{(\xi^N)^n:n\in \mathbb Z\}}$ is dense in $\mathbb T$ it follows that this action is minimal. Thus ${(\mathcal A_\theta, G)}$ is strongly minimal.  
  \end{example}
  
\section{The non-unital setting}\label{non-unital-section}

In this section we examine the extent to which some of the results in the previous section hold when the underlying C*-algebra in the C*-dynamical system is not assumed to be unital. Many results from the previous section hold with out much extra effort and we indicate which results those are here.
\medskip

Unlike the case where the underlying C*-algebra is unital, in this section we will refer to ${(\A, G)}$ where $\A$ is potentially non-unital as a C*-dynamical system instead of a C*-flow (continuing the analogy that in the classical setting the term flow is reserved for compact dynamical systems).

\subsection{Strong minimality characterizations}

The goal of this section is to investigate to what extent Theorem \ref{strong-minimality-char} holds. The first thing to note is that in the non-unital setting the connection to the Furstenberg boundary disappears in general when the condition that the map be unital is replaced with the map being nonzero.

\begin{example}
    Let $G$ be any infinite group and let ${\A=c_0(G\times G)}$. When equipped with the diagonal action by left translation ${(\A, G)}$ is a C*-dynamical system which is not (strongly) minimal. However, there is no nonzero $G$-equivariant positive map ${\A\to C(\partial_F G)}$. Indeed, it is easy to check using minimality of ${\partial_F G}$ that there is no nonzero $G$-equivariant positive map ${c_0(G)\to C(\partial_F G)}$, so it follows that any positive map ${c_0(G\times G)\to C(\partial_F G)}$ must also be zero.
\end{example}

Despite ${C(\partial_F G)}$ not playing the same role there is still a single C*-flow which can be used to check for strong minimality, namely ${\ell^\infty (G)}$. The key to ${\ell^\infty(G)}$ working as a replacement for ${C(\partial_F G)}$ in the non-unital setting are the Poisson maps, namely Proposition \ref{strong-minimality-poisson-map-char} still holds in this setting.

\begin{prop}\label{non-unital-strong-minimality-poisson-map-char}
    A C*-dynamical system ${(\mathcal A, G)}$ is strongly minimal if and only if the Poisson map ${\mathcal P_\psi:\mathcal A\to \ell^\infty(G)}$ is faithful for every pure state in $\psi$.
\end{prop}

\begin{proof}
    The proof of Proposition \ref{strong-minimality-poisson-map-char} does not use the unit anywhere.
\end{proof}

An easy consequence of Proposition \ref{non-unital-strong-minimality-poisson-map-char} is that once the appropriate modifications are made Proposition \ref{strong-minimality-extrinsic-char} holds in the non-unital setting.

\begin{cor}\label{non-unital-strong-minimality-extrinsic-char}
    Let ${(\A,G)}$ be a C*-dynamical system. The following are equivalent: 
\begin{enumerate}
\item ${(\A,G)}$ is strongly minimal. 
\item Every $G$-equivariant positive linear map ${\Phi:\A\to\B}$ is either faithful or identically zero, for every C*-dynamical system ${(\B,G)}$. 
\item Every $G$-equivariant completely positive linear map ${\Phi:\A\to\B}$ is faithful, for every C*-dynamical system ${(\B,G)}$. 
\item Every $G$-equivariant positive linear map ${\Phi:\A\to \ell^\infty(G)}$ is faithful. 
\end{enumerate}
\end{cor}

Proposition \ref{strong-minimality-hereditary-minimality-char} holds in the non-unital setting without any modification.

\begin{prop}
    A C*-dynamical system ${(\A,G)}$ is strongly minimal if and only if every $G$-invariant C*-subalgebra ${\B\subseteq\A}$ is minimal.
\end{prop}

\begin{proof}
    It is clear by using Corollary \ref{non-unital-strong-minimality-extrinsic-char} that if ${(\A, G)}$ is strongly minimal and ${\B\subseteq \A}$ is a $G$-invariant C*-subalgebra then ${(\B, G)}$ is strongly minimal, and hence minimal. Conversely, Lemma \ref{completely-positive-map-faithful-domain} holds when the condition that the map be unital and completely positive is replaced with the condition that it is a contractive completely positive map (with the same proof), so the same reasoning as in the proof of Proposition \ref{strong-minimality-hereditary-minimality-char} works here.
\end{proof}

Finally, Proposition \ref{strong-minimality-local-char} partially holds in the non-unital setting.

\begin{prop}
    If ${(\A, G)}$ is a C*-dynamical system then ${(\A, G)}$ is strongly minimal if and only if for every nonzero open projection ${e\in \A^{**}}$ and compact projection ${f\in\A^{**}}$ there are ${s_1,\dots, s_n\in G}$ with ${f\leq \smash{\bigvee_{j\,=\,1}^n s_j\msp e}}$.
\end{prop}

\begin{proof}
    The proof follows immediately from the comments in Remark \ref{open-proj-covers-one} and the fact that every nonzero closed projection dominates minimal projection in ${\A^{**}}$ (which is necessarily compact).
\end{proof}

Summarizing the results of this section yield the following non-unital version of Theorem \ref{strong-minimality-char}.

\begin{theorem}
    Let ${(\A,G)}$ be a C*-dynamical system. The following are equivalent: 
\begin{enumerate}
\item The C*-dynamical system ${(\A, G)}$ is strongly minimal. 
\item Every $G$-invariant C*-subalgebra ${\B\subseteq\A}$ is minimal. 
\item Every nonzero $G$-equivariant positive linear map ${\Phi:\A\to \ell^\infty(G)}$ is faithful.
\item Every nonzero open projection ${e\in\A^{**}}$ and compact projection ${f\in \A^{**}}$ admit elements ${s_1,\dots,s_n\in G}$ such that ${\smash{\bigvee_{j\,=\,1}^n s_j\msp e\geq f}}$. 
\item The Poisson map ${\P_\psi:\A\to\ell^\infty(G)}$ is faithful for every pure state ${\psi\in P(\A)}$. 
\end{enumerate}
\end{theorem}

\subsection{Subhomogeneous C*-algebras}

In the non-unital setting the natural analogue of Theorem \ref{subhomogeneous-minimality} holds without much modification to the proof.

\begin{theorem}\label{non-unital-subhomogeneous}
    Let ${(\mathcal A, G)}$ be a C*-dynamical system. If $\mathcal A$ is subhomogeneous of degree $n$, then ${(\mathcal A, G)}$ is strongly minimal if and only if the following three conditions hold:

      \begin{enumerate}
          \item $\mathcal A$ is homogeneous;
          \item The action of $G$ on the locally compact space ${\widehat{\mathcal A}}$ is minimal;
          \item There are no nontrivial $G$-invariant projections in ${M(\mathcal A)}$, the multiplier algebra of $\mathcal A$.
      \end{enumerate}
\end{theorem}

\begin{proof}
    Most aspects of the proof go through verbatim given that the lemmas preceding Theorem \ref{subhomogeneous-minimality} are valid without the assumption that the C*-algebras are unital. The only parts that need to be addressed are those that involve the multiplier algebra of $\mathcal A$. If ${(\mathcal A, G)}$ is strongly minimal then ${M(\mathcal A)}$ cannot have any $G$-invariant projections as any projection in ${M(\mathcal A)}$ is both open and closed. If $\mathcal A$ is homogeneous then ${X_\mathcal A}$ is a locally compact Hausdorff space and the multiplier algebra of $\mathcal A$ can be identified with the functions in ${C_b(X_\mathcal A, M_n)}$ which are ${ U(M_n)}$-equivariant. Thus, if $\mathcal A$ is homogeneous and the action of $G$ on ${\widehat{\mathcal A}}$ is minimal then as in the proof of Theorem \ref{subhomogeneous-minimality} the additional assumption that ${(A, G)}$ is not strongly minimal implies that there is a nontrivial $G$-invariant projection in ${C_b(X_\mathcal A, M_n)}$ which is ${ U(M_n)}$-equivariant and has constant rank.
\end{proof}

\begin{remark}
Remark \ref{subhomogeneous-remark} still holds in this setting: condition (3) in Theorem \ref{non-unital-subhomogeneous} can be replaced with the weaker condition that there is no $G$-invaraint projection in ${C_0(X_\A, M_n)}$ which is ${U(M_n)}$-equivariant and has constant rank.
\end{remark}

Corollary \ref{MinHomogSimp} also holds in the non-unital setting with the appropriate modifications.

\begin{cor}\label{non-unital-subhomogeneous-simplification}
    Let $X$ be a locally compact Hausdorff space and ${n\in\Nb}$. If ${(C_0(X,M_n), G)}$ is a C*-dynamical system, then ${(C_0(X,M_n), G)}$ is strongly minimal if and only if the following conditions hold:

      \begin{enumerate}
          \item The induced action of $G$ on $X$ is minimal;
          \item There is no nontrivial $G$-invariant projection in ${C_b(X,M_n)}$ of constant rank.
      \end{enumerate}
\end{cor}

Using Corollary \ref{non-unital-subhomogeneous-simplification} one can also produce examples as in Example \ref{proximal-example}.

\begin{example}
    Suppose that $G$ is a group which acts on a locally compact space $X$ minimally with the following proximality property: there is a net ${(s_i)}$ in $G$ and a point ${y\in X}$ so that ${s_ix\to y}$ for every ${x\in X}$. Note that this property always holds for proximal actions on compact spaces. Now, let ${\alpha:G\to \Aut(M_n)}$ be an action of $G$ on ${M_n}$ which is strongly minimal. If one defines the natural action of $G$ on ${C_0(X, M_n)}$ as in Example \ref{proximal-example}, then ${(C_0(X, M_n), G)}$ is strongly minimal. Applying Corollary \ref{non-unital-subhomogeneous-simplification} it suffices to check that there is no nontrivial $G$-invariant function in ${C_b(X, M_n)}$. Up to passing to a subnet of ${(s_i)}$, it can be assumed that ${\alpha(s_i^{-1})\to \beta\in \Aut(M_n)}$. If ${f\in C_b(X, M_n)}$, then for every ${x\in X}$
    \eqnsp{f(x)=\alpha(s_i^{-1})(f(s_ix))\to \beta(f(y))}
    In particular, ${\beta(f(y))=f(y)}$, so $f$ is constant. A non-trivial example ($X$ non-compact and $n\geq 2$) with these conditions achieved can be constructed as follows: the group $K$ of affine automorphisms of ${X=\mathbb R^k}$ acts minimally (translations) and satisfies the proximality property (taking ${s_k=\frac{1}{k}\text{id}_{\mathbb R^n}}$ one has ${s_kx\to 0})$. While $K$ has no finite-dimensional irreducible representations which are not one-dimensional, if one takes ${G=\mathbb F_\infty}$ to be the free group on countably many generators and a group homomorphism ${G\to K}$ with dense image, then any irreducible representation of $G$ yields a strongly minimal action as above.
\end{example}

\end{document}